\documentclass{article}
\usepackage[T1]{fontenc}
\usepackage{amsthm, fullpage}
\usepackage{amsmath}
\usepackage{amssymb}
\usepackage{amsfonts}
\usepackage{eucal}
\usepackage[all]{xy}
\usepackage[frenchb, english]{babel}
\usepackage{pslatex}
\usepackage[pagebackref]{hyperref}

\newtheorem{prop}{Proposition}[subsection]
\newtheorem{theo}[prop]{Théor\`eme}
\newtheorem*{theo**}{Théorème}
\newtheorem{coro}[prop]{Corollaire}
\newtheorem*{conj*}{Conjecture}
\newtheorem{lemm}[prop]{Lemme}
\newtheorem{lemm*}{Lemme}[prop]

\theoremstyle{definition}

\newtheorem{vide}[prop]{}
\newtheorem{defi}[prop]{Définition}
\newtheorem*{defi*}{Définition}

\theoremstyle{remark}
\newtheorem{rema}[prop]{Remarques}

\newtheorem{nota}[prop]{Notations}

\numberwithin{equation}{prop}

\newcommand{\riso}{ \overset{\sim}{\longrightarrow}\, }
\newcommand{\liso}{ \overset{\sim}{\longleftarrow}\, }

\newcommand{\Spf}{\mathrm{Spf}\,}

\renewcommand{\sp}{\mathrm{sp}}

\newcommand{\FF}{{\mathcal{F}}}

\newcommand{\E}{{\mathcal{E}}}
\newcommand{\G}{{\mathcal{G}}}
\renewcommand{\H}{{\mathcal{H}}}

\newcommand{\D}{{\mathcal{D}}}
\newcommand{\I}{{\mathcal{I}}}
\newcommand{\PP}{{\mathcal{P}}}

\renewcommand{\O}{{\mathcal{O}}}

\newcommand{\V}{\mathcal{V}}
\renewcommand{\S}{\mathcal{S}}

\newcommand{\Y}{\mathcal{Y}}
\newcommand{\ZZ}{\mathcal{Z}}
\newcommand{\X}{\mathfrak{X}}

\newcommand{\U}{\mathfrak{U}}

\newcommand{\A}{\mathbb{A}}
\renewcommand{\P}{\mathbb{P}}

\newcommand{\DD}{\mathbb{D}}
\renewcommand{\L}{\mathbb{L}}
\newcommand{\R}{\mathbb{R}}
\newcommand{\Q}{\mathbb{Q}}
\newcommand{\Z}{\mathbb{Z}}
\newcommand{\N}{\mathbb{N}}

\newcommand{\hdag}{  \phantom{}{^{\dag} }    }

\begin{document}
\title{Stabilité de l'holonomie sur les variétés quasi-projectives}
\author{Daniel Caro \footnote{L'auteur a bénéficié du soutien du réseau européen TMR \textit{Arithmetic Algebraic Geometry}
(contrat numéro UE MRTN-CT-2003-504917).}}
\date{}

\maketitle

\begin{abstract}
\selectlanguage{english}
Let $\V$ be a mixed characteristic complete discrete valuation ring with perfect residue field $k$.
We solve Berthelot's conjectures on the stability of the holonomicity over smooth projective formal $\V$-schemes.
Then we build a category of complexes of arithmetic $\D$-modules over quasi-projective $k$-varieties with bounded, $F$-holonomic cohomology. We get its stability under Grothendieck's six operations.

\end{abstract}

\selectlanguage{frenchb}

\tableofcontents

\section*{Introduction}

Afin d'obtenir une cohomologie $p$-adique sur les variétés algébriques sur un corps de caractéristique $p$
stable par les six opérations cohomologiques de Grothendieck (i.e., image directe, image directe extraordinaire,
image inverse, image inverse extraordinaire, produit tensoriel, foncteur dual),
Berthelot a construit une version arithmétique de la théorie des modules sur le faisceau des opérateurs différentiels 
(voir l'introduction
\cite{Beintro2} puis dans l'ordre \cite{Be0}, \cite{Be1}, \cite{Be2}).
Rappelons quelques éléments de sa théorie :
soient $\V$ un anneau de valuation discret complet d'inégales caractéristiques $(0,p)$,
de corps résiduel $k$ supposé parfait et de corps de fraction $K$,
$\PP$ un $\V$-schéma formel lisse et $P$ sa fibre spéciale.
Dans la version arithmétique de Berthelot, le faisceau des opérateurs différentiels usuels $\D$ est remplacé par $\D ^{\dag} _{\Q}$.
Plus précisément, il construit le faisceau sur $\PP$ des opérateurs différentiels de niveau fini et d'ordre infini
noté $\D ^\dag _{\PP,\Q}$ ; ce dernier correspondant à la tensorisation par $\Q$ (indiqué par l'indice $\Q$) du complété faible $p$-adique (indiqué par le symbole {\og $\dag$\fg})
du faisceau classique $\D  _{\PP}$ des opérateurs différentiels sur $\PP$.
On dispose de plus de la notion de $F\text{-}\D ^\dag _{\PP,\Q} $-module, i.e.,
de $\D ^\dag _{\PP,\Q} $-module $\E$
muni d'une structure de Frobenius, i.e. d'un isomorphisme $\D ^\dag _{\PP,\Q} $-linéaire
de la forme $F ^* (\E) \riso \E$ avec $F^*$ désignant l'image inverse par l'endomorphisme (ou une puissance) du Frobenius absolu 
$P\to P$.
Il a aussi obtenu une notion de $F\text{-}\D ^\dag _{\PP,\Q} $-module holonome
en s'inspirant du cas classique : un $F\text{-}\D ^\dag _{\PP,\Q} $-module cohérent est holonome
s'il est nul ou si la dimension de sa variété caractéristique est égale à la dimension de $P$.
Il conjecture surtout que les $F$-complexes de $\D ^\dag _{\PP,\Q} $-modules à cohomologie bornée et holonomes
sont stables par opérations cohomologiques (de manière analogue à ce qui se passe dans la théorie classique).
Malgré la résolution de cette stabilité dans quelques cas importants (e.g., image inverse extraordinaire par un morphisme lisse, produit tensoriel externe, foncteur dual), voici ce qu'il restait à établir : 
la stabilité de l'holonomie par produit tensoriel interne,
image directe et image inverse extraordinaire 
 (voir les conjectures \cite[6.3.6]{Beintro2}).
Pour bénéficier de coefficients $p$-adiques stables,
une autre approche a été de définir les $\D$-modules arithmétiques surholonomes
(\cite{caro_surholonome}).
 D'après Caro et Tsuzuki (voir \cite{caro-Tsuzuki-2008}),
 les $F$-complexes de $\D$-modules arithmétiques surholonomes
 sont stables par les six opérations cohomologiques de Grothendieck.
 De plus, les $F$-complexes de $\D$-modules arithmétiques surholonomes sont holonomes. 
 L'égalité entre ces deux notions équivaut d'ailleurs à la validation des conjectures ci-dessus sur la stabilité de l'holonomie (voir \cite[8.2]{caro_surholonome})).
Ce travail s'attaque à ce problème dans le cas où le $\V$-schéma formel lisse (dans lequel se plonge une $k$-variété où vivent véritablement les complexes de $\D$-modules arithmétiques) peut être choisi projectif. 

Abordons à présent le contenu de ce travail.
  Soient $\PP$ un $\V$-schéma formel propre et lisse,
  $H _0$ un diviseur ample de $P _0$, $\mathfrak{A}$ l'ouvert de $\PP$ complémentaire de $H _0$.
  En fait, comme le diviseur est ample, il ne coûte pas grand chose de supposer que $\PP$ est le complété $p$-adique de l'espace projectif sur $\V$
et $H _{0}$ est un hyperplan défini par une des coordonnées canoniques.
  
Soit $\E$ un $\D ^\dag _{\PP} (\hdag H _0) _\Q$-module cohérent. 
Le résultat principal de la première partie de ce travail (voir \ref{pass-cvasurcv} et la preuve de \ref{conjD-proj}) est que, quitte à rajouter des singularités surconvergentes (i.e. à augmenter $H _{0}$), $\E$ est associé à un isocristal surconvergent sur $Y$, 
où $Y$ est un ouvert affine et lisse (dense dans une composante irréductible)
du support de $\E$.
Remarquons que, pour donner un sens à ce résultat, la structure de Frobenius est superflue grâce aux récents travaux (\cite{caro-pleine-fidelite}).
Cela nous a permis d'améliorer et surtout de simplifier (par rapport à une version antérieure prépubliée de ce travail qui n'utilisait pas \cite{caro-pleine-fidelite} car écrite avant) la preuve de ce résultat \ref{pass-cvasurcv} car on ne doit plus se préoccuper des compatibilités à Frobenius dans les constructions. 
Le premier ingrédient technique permettant d'éviter le problème du support de $\E$ non lisse
(la stabilité de la surholonomie nous permettait d'utiliser le théorème
de désingularisation de de Jong ; cette méthode ne fonctionne plus a priori et il faut donc innover)
est l'utilisation des opérateurs différentiels {\og à la Mebkhout-Narvaez\fg}
(voir \cite{Mebkhout-Narvaez-Macarro_D}) et le théorème de comparaison de Noot-Huyghe
(qui n'est validé a priori que pour les $\V$-schémas formels projectifs lisses).
C'est l'unique raison technique pour laquelle nous nous sommes restreint aux $\V$-schémas formels projectifs et lisses.
Le second ingrédient technique dû à Kedlaya (voir \cite{Kedlaya-coveraffinebis}) qui nous permet de nous ramener au cas déjà traité où le support de $\E$ est lisse (voir \cite{caro-construction}) est le fait que tout point fermé de $\PP$
possède un voisinage ouvert qui soit fini et étale sur un espace affine (ce dernier possède de façon remarquable une compactification lisse).

Soit $ \E \in F \text{-}D ^\mathrm{b} _\mathrm{coh} (\D ^\dag _{\PP} (\hdag H _0) _\Q)$.
  Nous prouvons dans le deuxième chapitre de ce papier (voir \ref{conjD-proj}) que
  si   
  $\E |\mathfrak{A} \in F \text{-}D ^\mathrm{b} _\mathrm{hol} (\D ^\dag _{\mathfrak{A},\Q})$
  alors
   $\E \in F \text{-}D ^\mathrm{b} _\mathrm{surhol} (\D ^\dag _{\PP,\Q})$, i.e., $\E$ est un $F$-complexe surholonome s'il est holonome en dehors des singularités surconvergentes. Ainsi, dans le cas d'un diviseur ample, cela répond positivement à la conjecture la plus forte de Berthelot de \cite[6.3.6]{Beintro2}. 
Donnons une esquisse de sa preuve : l'idée est de procéder par récurrence sur la dimension du support noté $X$ de $\E$. On peut en outre supposer que $\E$ est un module. 
Or, d'après la première partie, il existe un ouvert affine et lisse $Y$ (dense dans une composante irréductible) de $X$ tel que 
$\E$ soit associé à un isocristal surconvergent sur $Y$.
Comme les $F$-isocristaux surconvergents sur les $k$-variétés lisses sont surholonomes (voir \cite{caro-Tsuzuki-2008}), nous concluons la récurrence. 
Remarquons de plus que, comme il existe des isocristaux surconvergents sur des $k$-variétés lisses qui ne sont pas cohérents ni a fortiori surholonomes, 
le théorème \ref{conjD-proj} énoncé ci-dessus est faux sans structure de Frobenius (voir \ref{ss-Frob-rema}).
Nous avons donc besoin in fine de la structure de Frobenius pour valider la récurrence. 
Enfin, notons que la preuve de la surholonomie des $F$-isocristaux surconvergents de \cite{caro-Tsuzuki-2008} utilise notamment la stabilité de la surholonomie par image directe par un morphisme propre, ce qui empêche a priori de vérifier directement (i.e. sans la notion de surholonomie) 
l'holonomie des $F$-isocristaux surconvergents sur les $k$-variétés lisses.

Le théorème \ref{conjD-proj} implique que l'image directe (resp. l'image inverse extraordinaire) par un morphisme de 
$\V$-schémas formels projectifs et lisses  d'un $F$-complexe holonome est un $F$-complexe holonome. 
Cette stabilité nous permet de construire les complexes à cohomologie bornée et $F$-holonome
de $\D$-modules arithmétiques sur les $k$-variétés quasi-projectives.
Comme c'est le cas pour les complexes surholonomes, on obtient ensuite la stabilité par les six opérations cohomologiques de Grothendieck
de cette catégorie de coefficients $p$-adiques. 

\bigskip

{\bf Notations.}
Soient $\V$ un anneau de valuation discret complet d'inégales caractéristiques $(0,p)$, 
$\pi$ une uniformisante, 
$k$ son corps résiduel supposé parfait, $K$ son corps de fraction.
On fixe $s$ un entier et $\sigma \,:\, \V \riso \V$ un relèvement de la puissance $s$-ième de Frobenius.
Enfin, $m$ désignera par défaut un entier positif.

$\bullet$ Si $M$ est un $\V$-module (resp. un groupe abélien),
on pose $M _K := M \otimes _\V K$ (resp. $M _\Q=M \otimes _\Z \Q$).
Un module est par défaut un module à gauche.

$\bullet$
Pour tout $\V$-schéma formel faible lisse $X ^\dag $, 
on désigne par $X _{0}$ sa fibre spéciale, $\X$ le $\V$-schéma formel déduit par complétion $p$-adique. 
On note $\D ^{(m)} _{X ^\dag}$,
le faisceau des opérateurs différentiels d'ordre fini et de niveau $m$ sur $X ^\dag$ (e.g., voir \cite[2]{caro_devissge_surcoh}).
Si $t _{1}, \dots, t _{d}$ sont des coordonnées locales de $X ^{\dag}$ et $\partial _{1}, \dots, \partial _{d}$ sont les dérivations
correspondantes, on désignera comme d'habitude (voir \cite{Be1})
la base canonique de $\D ^{(m)} _{X ^\dag}$ par
$\underline{\partial} ^{<\underline{k}> _{(m)}}$.

Pour tout ouvert affine $U ^\dag $ de $X ^\dag$,
on pose $D ^{(m)} _{U ^\dag}=\Gamma (U^\dag , \D ^{(m)} _{X ^\dag})$ et 
$D _{U ^\dag,K}=\Gamma (U^\dag , \D ^{(m)} _{X ^\dag}) _{K}$ (qui ne dépend canoniquement pas de $m$).
On note $D ^{(m)\dag} _{U ^\dag}$ la complétion $p$-adique faible comme $\V$-algèbre de $D ^{(m)} _{U ^\dag}$.
Enfin, on pose
$D ^{\dag} _{U ^\dag} :=
\underset{\longrightarrow}{\lim}\, _m D ^{(m) \dag} _{U ^\dag}$.

$\bullet $ Pour tout $\V$-schéma formel lisse $\X$, 
on note $X _{0}$ sa fibre spéciale.
Pour tout ouvert affine $\U$ de $\X$,
on pose
$\widehat{D} ^{(m)} _{\U}:= \Gamma (\U, \widehat{\D} ^{(m)} _{\U})$,
$D ^\dag _{\U}:= \Gamma (\U, \D ^\dag _{\X})$.

$\bullet$ Si l'entier $n$ ne pose aucune ambiguïté,
on écrira plus simplement $D _{K}$ (resp. $D ^{(m)\dag} $, resp. $D ^\dag $, resp. $\widehat{D} ^{(m)}$) à la place de
$D _{\A ^{n\dag} _\V ,K}$
(resp. $D ^{(m)\dag} _{\A ^{n\dag} _\V}$, resp. $D ^{\dag} _{\A ^{n\dag} _\V}$, resp. $\widehat{D} ^{(m)} _{\widehat{\A} ^{n} _\V}$).

$\bullet$ Nous reprenons les notations usuelles des opérations cohomologiques que nous ne rappelons pas (e.g., voir \cite{Beintro2} ou le premier chapitre de \cite{caro_surholonome}).

\section{Caractérisations des isocristaux surconvergents sur les schémas finis étales sur l'espace affine}
Soit $X ^{\dag}$ un $\V$-schéma formel faible lisse. Grâce à Kedlaya (voir \cite{Kedlaya-coveraffine} ou \cite{Kedlaya-coveraffinebis}), 
il existe un ouvert affine dense $U ^{\dag}$ de $X ^{\dag}$ et un morphisme fini et étale de la forme $g\,:\,U^\dag \rightarrow \A ^{n\dag} _\V$.
Nous étudions dans cette section les opérateurs différentiels de Mebkhout-Narvaez  (voir la fin de \cite{Mebkhout-Narvaez-Macarro_D} et les théorèmes de comparaison de Noot-Huyghe avec la théorie de Berthelot des $\D$-modules arithmétiques de \cite{huyghe-comparaison}) sur un tel $U ^{\dag}$. Nous en déduirons une description des isocristaux surconvergents sur $U ^{\dag}$ (voir \ref{Eq-isoc-Ddag}).
Après ce travail préliminaire, 
ce chapitre culminera avec le théorème \ref{pass-cvasurcv} qui donne une condition suffisante pour qu'un $\D$-module arithmétique soit un isocristal surcohérent (voir \cite[6]{caro-2006-surcoh-surcv} et pour la version sans structure de Frobenius \cite{caro-pleine-fidelite}). Grosso modo, sur certains sous-schémas de l'espace projectif, dans la définition de la notion d'isocristal surcohérent, on peut affaiblir les hypothèses en remplaçant 
surcohérent par cohérent. 

\subsection{Compléments sur les opérateurs différentiels de Mebkhout-Narvaez}

Nous vérifions ici que le foncteur section globale appliqué au faisceau des opérateurs différentiels de Mebkhout-Narvaez commutent à l'extension 
d'un morphisme fini étale sur un espace affine (voir \ref{AotimeDdag=DdagU}). 
On en déduit facilement \ref{oub-com-otimes}. Ces résultats nous serviront pour valider la caractérisation \ref{pre-Eq-isoc-Ddag} des isocristaux surconvergents sur les schémas finis étales sur l'espace affine ou pour établir
le théorème \ref{pass-cvasurcv} de la section suivante.

Rappelons la définition de la complétion $p$-adique faible donnée par Noot-Huyghe dans \cite[1.3]{huyghe-comparaison}
dans le cas d'une $\V$-algèbre non nécessairement commutative:
\begin{defi}
\label{Adagnoncom}
  Pour tout entier $N$,
on note $B _N$ l'algèbre des polynômes à $N$ indéterminées sur $\V$ non commutative
(i.e. la $\V$-algèbre tensorielle de $\V ^N$).
Soit $A$ une $\V$-algèbre non nécessairement commutative.
On note $\widehat{A}$ la complétion $p$-adique de $A$ et
$A ^\dag$ le sous-ensemble de $\widehat{A}$ des éléments $z$
tels qu'il existe une constante $c\in \R $,
des éléments $x_1, \dots,x _n \in A$ et, pour tout $j\in \N$,
des polynômes $P _j \in \pi ^j B _n$ tels que
$\deg P _j \leq c (j+1)$ et
\begin{equation}
  \notag
  z = \sum _{j\in \N } P _j (x _1,\dots, x_n).
\end{equation}
L'ensemble $A ^\dag$ est une sous-$\V$-algèbre de $\widehat{A}$
et se nomme {\og la complétion $p$-adique faible de $A$ en tant que $\V$-algèbre\fg}
ou {\og la complétion $p$-adique faible de $A$\fg} s'il n'y a aucune ambiguïté sur l'algèbre de base $\V$
(dans notre travail, ce sera toujours sur $\V$).
On dira aussi que
{\og $z$ est engendré de manière faiblement complète sur $\V$ par
les éléments $x _1,\dots, x _n$ de $A$\fg}.

\end{defi}

\begin{rema}
\label{remavpi=j}
  Avec les notations de \ref{Adagnoncom}, supposons $c \geq 1$.
  Notons $R _0 =0$ et, pour tout $j\in \N$,
  définissons par récurrence sur $j\geq 0 $ les polynômes $Q _j$ et $R _{j+1}$ en posant :
  $R _j + P _j = Q _j + R _{j+1}$ où
  $Q _j$ désigne la somme des monômes de $R _j + P _j $ dont la valuation $\pi$-adique du
  coefficient vaut $j$ tandis que $R _{j+1}$ est la somme des autres termes.
  On dispose de l'égalité $z = \sum _{j\in \N } Q _j (x _1,\dots, x_n)$ avec
  $\deg Q _j \leq c (j+1)$ et tous les monômes de $Q _j $
ont une valuation $\pi$-adique égale à $j$
(on remarque aussi que $\deg R _{j+1} \leq c (j+1)$ et
  $R _{j+1} \in \pi ^{j+1} B _n$).
\end{rema}

\begin{lemm}
\label{Adagest-fc}
  Soient $A$ une $\V$-algèbre, $c\in \R$,
  $y _1, \dots, y _n \in A ^\dag$,
  et, pour tout $j\in \N$,
des polynômes $P _j \in \pi ^j B _n$ tels que
$\deg P _j \leq c (j+1)$.
Alors l'élément $z : =\sum _{j\in \N } P _j (y _1,\dots, y_n) $ appartient à $A ^\dag$.
Plus précisément, si $y _1, \dots, y_n$ sont engendrés de manière faiblement complète sur $\V$
par des éléments $ x _1,\dots, x _m$ de $A$,
alors $z$ est engendré de manière faiblement complète sur $\V$ par $ x _1,\dots, x _m$ .
En particulier, on dispose de l'égalité $A ^\dag=(A ^\dag)^\dag$.
\end{lemm}

\begin{proof}
  Soit $P (Y _{1},\dots,Y _{n})= \lambda \cdot Y _{\phi (1)}  Y _{\phi (2)} \cdots Y _{\phi (r)} $ avec $\phi (1),\dots, \phi (r)\in \{1,\dots, n \}$
  un monôme de $P _j$ avec $\lambda \in \V$.
  Par hypothèse, $v _\pi (\lambda) \geq j$ et
  $r \leq c (j+1)$.
Quitte à augmenter $c$,
il existe $x _1, \dots, x _m \in A$ tels que,
pour tout $i ' =1,\dots, n$, pour tout $k\in \N$,
il existe des polynômes $P _{i',k} \in \pi ^k B _m$ tels que
$\deg P _{i',k} \leq c (k+1)$
et 
$y _{i'} =\sum _{k\in \N } P _{i',k} (x _1,\dots, x_m)$.
On obtient
$$
P (y_1,\dots, y_n) =
\sum _{k _{1} \in \N,\dots, k _{r} \in \N }
\lambda \cdot P _{\phi (1),k _{1}}(x_1, \dots, x_m)  \cdots P _{\phi (r),k _{r}} (x _1,\dots, x_m).
$$
  Notons $P _{(k _{1},\dots, k _{r})} (X _1,\dots, X_m):=\lambda \cdot P _{\phi (1),k _{1}}(X_1, \dots, X_m)  \cdots P _{\phi (r),k _{r}} (X _1,\dots, X_m)$
   les polynômes de $B _{m}$ apparaissant dans cette somme.
   Notons $v _\pi (P _{(k _{1},\dots, k _{r})})$ la valuation $\pi$-adique minimale des coefficients de $P _{(k _{1},\dots, k _{r})}$.
   Par construction, $v _\pi (P _{(k _{1},\dots, k _{r})}) \geq j+ \sum _{i=1} ^r k_{i}$. 
   
   Soit $J$ un entier. Or, pour $j$ fixé, $P _{j}$ est une somme finie de tels monômes $P$. 
   Comme $v _\pi (P _{(k _{1},\dots, k _{r})}) \geq j+ \sum _{i=1} ^r k_{i}$, il en résulte qu'il existe un nombre fini de polynômes $P _{(k _{1},\dots, k _{r})}$ définis comme ci-dessus 
   et tels que $v _\pi (P _{(k _{1},\dots, k _{r})})=J$.
   Soit $Q _J (X_1, \dots, X_m)$ la somme de tous ces polynômes $P _{(k _{1},\dots, k _{r})}(X_1, \dots, X_m)$ tels que $v _\pi (P _{(k _{1},\dots, k _{r})}) =J$. 
   Alors $v _\pi (Q _J) \geq J$ et 
$z = \sum _{J\in\N} Q _J (x_1,\dots, x _m)$.

Il reste à majorer le degré de $Q _J$ : 
soit $P _{(k _{1},\dots, k _{r})}$ un de ces polynômes 
   tels que $v _\pi (P _{(k _{1},\dots, k _{r})})=J$.
   On a alors $\deg P _{(k _{1},\dots, k _{r})} = \sum _{i=1} ^r  \deg P _{\phi (i),k _{i}}
   \leq c \sum _{i=1} ^r  ( k _{i} +1) $.
Comme $r \leq c (j+1)$, comme $J \geq j+ \sum _{i=1} ^r k_{i}$, on en déduit alors 
   $\deg P _{(k _{1},\dots, k _{r})} \leq c (J -j + c (j+1)) \leq c (J +1 + c (J +1))= c (1 +c) (J +1)$.
   Posons $C =c (1+c)$. On a donc établi $\deg Q _J \leq C (J+1)$.
   D'où le résultat.
\end{proof}

Le lemme ci-après est immédiat :
\begin{lemm}
  \label{proddagconst}
  Soient $N _1 , N _2 \in \N $, $c \in \R$ et, pour tout $j\in \N$,
  soient $P_j, Q _j \in \pi ^j B_n$ tels que $\deg P _j \leq c (j + N _1)$,
  $\deg Q _j \leq c (j + N _2)$.
  On définit des éléments de $\V [t_1,\dots, t_n] ^\dag$ en posant :
  $z = \sum _{j\in \N} P _j (t_1,\dots, t_n)$,
  $u = \sum _{j\in \N} Q _j (t_1,\dots, t_n)$.

\begin{enumerate}
  \item \label{proddagconst-1)} Il existe, pour tout $\underline{k} \in \N ^n$ et tout $j\in\N$, des polynômes $\widetilde{P}_j\in \pi ^j B_n$ tels que
$\deg \widetilde{P} _j \leq c (j + N_1)$ et
$\underline{\partial} ^{<\underline{k}> _{(m)}}(z) = \sum _{j\in \N} \widetilde{P} _j (t_1,\dots, t_n)$.

\item   \label{proddagconst-2)}Il existe, pour tout $j\in\N$, des polynômes $R _{j},\widetilde{R} _j \in \pi ^j B_n$ tels que
$\deg R _j \leq c (j + \max \{ N_1,N_2\})$, $\deg \widetilde{R} _j \leq c (j + N_1 +N _2)$ et 
$z+u = \sum _{j\in \N} R _j (t_1,\dots, t_n)$,
$z u = \sum _{j\in \N} \widetilde{R} _j (t_1,\dots, t_n)$.
\end{enumerate}
\end{lemm}

\begin{lemm}
\label{BerAotimeDdag=DdagU}
  Soient $f\,:\,\X' \rightarrow \X$ un morphisme fini étale de $\V$-schémas formels lisses,
  $\Y$ un ouvert affine de $\X$ muni de coordonnées locales et $\Y':= f^{-1} (\Y)$,
  $B:= \Gamma (\Y, \O _{\X}),\,
  B':= \Gamma (\Y', \O _{\X'})$.
  Le morphisme canonique
  $\D ^\dag _{\X'} \rightarrow f^*  \D ^\dag _{\X} $ est un isomorphisme.
  Le morphisme canonique $f^{-1}  \D ^\dag _{\X} \rightarrow \D ^\dag _{\X'} $
  induit est en fait un morphisme d'anneaux.
  Enfin, les morphismes qui en résultent
  $B' \otimes _B D ^\dag _{\X} \rightarrow D ^\dag _{\X'} $,
  $D ^\dag _{\X} \otimes _B B'  \rightarrow D ^\dag _{\X'} $
  sont des isomorphismes.
\end{lemm}

\begin{proof}
  Notons $f _{i}\,:\, X '_{i}\to X _{i}$ (resp. $B _{i}$, $B '_{i}$) la réduction modulo $\pi ^{i+1}$ de $f$ (resp. $B$, $B'$).
  Via un calcul en coordonnées locales, on vérifie que le morphisme canonique
  $\D ^{(m)} _{X '_{i}} \rightarrow f^*  \D ^{(m)} _{X _{i}} $ est un isomorphisme.
  Par un calcul analogue, on prouve que le morphisme $f^{-1}  \D ^{(m)} _{X _{i}} \rightarrow \D ^{(m)} _{X _{i}'} $,
 induit via son inverse, est en fait un morphisme d'anneaux et
que les morphismes 
  $B' _{i}\otimes _{B_{i}} D ^{(m)} _{X _{i}} \rightarrow D ^{(m)} _{X _{i}'} $,
  $D ^{(m)} _{X _{i}} \otimes _{B_{i}} B _{i} '  \rightarrow D ^{(m)} _{X _{i}'} $
  induits en prenant les sections globales
  sont des isomorphismes.
Par passage à la limite projective sur $i$, puis passage à la limite inductive sur le niveau $m$,
on en déduit le lemme. 
\end{proof}

En remplaçant les faisceaux des opérateurs différentiels de Berthelot par ceux de Mebkhout-Narvaez, 
la vérification de l'analogue du lemme \ref{BerAotimeDdag=DdagU} est techniquement plus délicate
car la complétion $p$-adique faible est moins maniable que la complétion $p$-adique. 
Avant de traiter la partie (voir \ref{AotimeDdag=DdagU}) moins aisée de cette analogie, 
donnons d'abord sa partie triviale :

\begin{lemm}
\label{lemm-AotimeDdag=DdagU}
Soit $g\,:\,U ^{\prime \dag } \rightarrow U ^{\dag }$ un morphisme fini étale de $\V$-schémas formels faibles affines et lisses.
Le morphisme canonique de $\D ^{(m)} _{U ^{\prime \dag }} $-modules à gauche
$\D ^{(m)} _{U ^{\prime \dag }} \rightarrow g ^* \D ^{(m)} _{U ^{\dag }}  $ est un isomorphisme.
On dispose de plus des morphismes canoniques de $\V$-algèbres
$D ^{(m)\dag} _{U ^{\dag }}\rightarrow D ^{(m)\dag} _{U ^{\prime \dag }}$,
$D ^\dag _{U ^{\dag }} \rightarrow D ^\dag _{U ^{\prime \dag }}$.
\end{lemm}

\begin{proof}
Comme $g$ est fini et étale, 
on vérifie via un calcul en coordonnées locales (identique à celui de la preuve de \ref{BerAotimeDdag=DdagU}) que 
$\D ^{(m)} _{U ^{\prime \dag }} \rightarrow g ^* \D ^{(m)} _{U ^{\dag }}  $
est un isomorphisme.
Via son inverse, on construite alors le morphisme canonique $g ^{-1} \D ^{(m)} _{U ^{\dag }}
\rightarrow \D ^{(m)} _{U ^{\prime \dag }} $. 
On établit ensuite par un calcul en coordonnées locales que celui-ci est en fait 
un morphisme de $\V$-algèbres. En prenant les sections globales,
on obtient un morphisme de $\V$-algèbres $D ^{(m)} _{U ^{\dag }}
\rightarrow D ^{(m)} _{U ^{\prime \dag }} $. D'où, par fonctorialité de la complétion
$p$-adique faible de $\V$-algèbres,
le morphisme $D ^{(m)\dag } _{U ^{\dag }} \rightarrow D ^{(m)\dag} _{U ^{\prime \dag }} $.
Par passage à la limite sur le niveau, cela donne le morphisme canonique de $\V$-algèbres :
$D ^\dag _{U^{\dag }} \rightarrow D ^\dag _{U^{\prime \dag }}$.
\end{proof}

\begin{prop}
\label{AotimeDdag=DdagU}
Soient $g\,:\,U^\dag \rightarrow \A ^{n\dag} _\V$ un morphisme fini étale de
$\V$-schémas formels faibles lisses et
$A ^\dag := \Gamma (U ^\dag , \O _{U ^\dag})$.

Les morphismes canoniques $A ^\dag$-linéaires induits (via \ref{lemm-AotimeDdag=DdagU})
\begin{gather}\label{AotimeDdag=DdagU1}
  A ^\dag \otimes _{\V [ \underline{t}] ^\dag }D ^{(m)\dag} \rightarrow D ^{(m)\dag} _{U^\dag },\hspace{2cm}
D ^{(m)\dag}  \otimes _{\V [ \underline{t}] ^\dag }A ^\dag \rightarrow D ^{(m)\dag} _{U^\dag },\\
\label{AotimeDdag=DdagU2}
A ^\dag \otimes _{\V [ \underline{t}] ^\dag }D ^\dag \rightarrow D ^\dag _{U^\dag },\hspace{2cm}
D ^\dag  \otimes _{\V [ \underline{t}] ^\dag }A ^\dag \rightarrow D ^\dag _{U^\dag }
\end{gather}
sont des isomorphismes.
\end{prop}

\begin{proof}
Par symétrie et par passage à la limite inductive sur le niveau,
contentons-nous de vérifier que le morphisme canonique 
$\theta \,:\, A ^\dag \otimes _{\V [ \underline{t}] ^\dag }D ^{(m)\dag} \rightarrow D ^{(m)\dag} _{U^\dag }$
est un isomorphisme.

Comme $A ^\dag$ est une $\V [ \underline{t}] ^\dag $-algèbre finie,
$\V \{ \underline{t}\} \otimes _{\V [ \underline{t}] ^\dag } A ^\dag \riso \widehat{A}$. 
Il en résulte le premier isomorphisme :
$A ^\dag \otimes _{\V [ \underline{t}] ^\dag }\widehat{D} ^{(m)}\riso \widehat{A} \otimes _{\V \{ \underline{t}\}} \widehat{D} ^{(m)}
\riso \widehat{A} \widehat{ \otimes} _{\V \{ \underline{t}\}} \widehat{D} ^{(m)}
\riso \widehat{D} ^{(m)} _{\U},$ le dernier isomorphisme résultant de \ref{BerAotimeDdag=DdagU} (voir sa preuve).
Cet isomorphisme composé s'inscrit dans le diagramme commutatif :
\begin{equation}
\label{Ddagm-fet-inj}
\xymatrix @R=0,3cm {
{A ^\dag \otimes _{\V [ \underline{t}] ^\dag }\widehat{D} ^{(m)}} 
\ar[r] ^-{\sim}
& 
{ \widehat{D} ^{(m)} _{\U}} 
\\ 
{A ^\dag \otimes _{\V [ \underline{t}] ^\dag }D ^{(m)\dag}} 
\ar[r] ^-{\theta}
\ar@{^{(}->}[u]
& 
{ D ^{(m)\dag} _{U ^{\dag}}.}
\ar@{^{(}->}[u]
 }
\end{equation}
Comme $A ^\dag $ est une extension plate de $\V [ \underline{t}] ^\dag $, comme
$D ^{(m)\dag} $ et $ D ^{(m)\dag} _{U ^{\dag}}$ sont séparés (pour la topologie $p$-adique), il en résulte que
les flèches verticales de \ref{Ddagm-fet-inj} sont injectives. 
On obtient ainsi l'injectivité de $\theta$.

Prouvons à présent la surjectivité de $\theta$ via les étapes suivantes:

{\it $0)$ Fixons quelques notations.}
Soient $x _1,\dots, x _s \in A ^\dag$ engendrant $A ^\dag$ comme $\V [ t_1,\dots, t_n] ^\dag$-module.
Notons $X$ le vecteur colonne de coordonnées $x _1,\dots, x_s$.

$\bullet $ Pour tout $\underline{a} = (a_1,\dots, a_n) \leq p^m$ (i.e. $a_1,\dots, a_n \leq p^m$),
soient
$A ^{(\underline{a})} = ( a ^{(\underline{a})} _{ij}) _{1\leq i,j\leq s}
\in M _s (\V [ t_1,\dots, t_n] ^\dag)$ tel que
\begin{equation}
  \label{E1}
  \underline{\partial} ^{<\underline{a}> _{(m)}} (X )
  =
  A ^{(\underline{a})} X,   
\end{equation}
où $\underline{\partial} ^{<\underline{a}> _{(m)}}  (X)$ désigne le vecteur colonne
de coordonnées $\underline{\partial} ^{<\underline{a}> _{(m)}}  (x _1),\dots,
\underline{\partial} ^{<\underline{a}> _{(m)}} ( x_s)$.
Il existe une constante réelle $c \geq 1$ et, pour tout $k\in\N$,
des polynômes $P _{ijk} ^{(\underline{a})} \in \pi ^k B _n$ tels que
$\deg P _{ijk} ^{(\underline{a})} \leq c (k+1)$ et
$a _{ij} ^{(\underline{a})} = \sum _{k\in\N} P _{ijk} ^{(\underline{a})} (t_1,\dots, t_n)$.

$\bullet $ Pour tout $1 \leq b \leq s$, soit
$B ^{(b)} = (b ^{(b)} _{ij}) _{1\leq i,j\leq s}
\in M _s (\V [ t_1,\dots, t_n] ^\dag)$ tel que
\begin{equation}
  \label{E2}
  x_b   \cdot X
  =
  B ^{(b)} X,
\end{equation}
où $x_b   \cdot X$ est le vecteur colonne de coordonnées
$x _b x _1, \dots , x _b x _s$.
Quitte à augmenter $c$, il existe pour tout $k\in\N$,
des polynômes $P _{ijk} ^{(b)} \in \pi ^k B _n$ tels que
$\deg P _{ijk} ^{(b)} \leq c (k+1)$ et vérifiant
$b _{ij} ^{(b)} = \sum _{k\in\N} P _{ijk} ^{(b)} (t_1,\dots, t_n)$.

{\it $1)$} Soit $z \in D ^{(m)\dag} _{U ^\dag}$.
Il existe alors des éléments $y _{1}, \dots, y _{e}$ de $D ^{(m)} _{U ^\dag}$ tels que $z$ soit engendré de manière faiblement complète par $y _{1},\dots, y _{e}$. 
Comme
$D ^{(m)} _{U ^\dag}$ est une $\V[t_1,\dots, t _n] ^\dag $-algèbre engendrée par
$x _1, \dots, x _s$ et par
$\underline{\partial} ^{<\underline{a}> _{(m)}}$ pour $\underline{a} \leq p^m$, 
$y _{1},\dots, y _{e}$ sont alors
engendrés de manière faiblement complète par $x _1, \dots, x _s$ où l'on choisit $x _{1}=1$, par $t _{1},\dots, t _{n}$ et par
$\underline{\partial} ^{<\underline{a}> _{(m)}}$ pour $\underline{a} \leq p^m$. 
Il découle alors de \ref{Adagest-fc} que $z$ est engendré de manière faiblement complète par $x _1, \dots, x _s$, par $t _{1},\dots, t _{n}$ et par
$\underline{\partial} ^{<\underline{a}> _{(m)}}$ pour $\underline{a} \leq p^m$. 
Posons $N := n +s + n (p ^m +1)$. 
Il existe ainsi,
pour tout $J\in \N$,
des polynômes $P _J \in \pi ^J B _N$ tels que
$\deg  P _J \leq c (J+1)$ et
$z =\sum _{J\in\N} P _J (t _1, \dots, t _n, x _1, \dots, x _s,
\underline{\partial} ^{<\underline{a}> _{(m)}}, \underline{a}\leq p^m)$.

{\it $2)$}
D'après les formules \cite[2.2.4.(ii) et (iv)]{Be1},
le passage de droite à gauche d'un polynôme en $t _1,\dots, t_n$
par rapport
à un opérateur de la forme
$\underline{\partial} ^{<\underline{a}> _{(m)}}$
n'augmente pas le degré en $t _{1}, \dots, t _{n}$.
On peut donc supposer que $P _{J}$ est une somme {\it finie} de monômes 
de la forme :
$$M _J =\pi ^J P (t_1,\dots, t _n)
 Q _1 ( \underline{\partial} ^{<\underline{a}> _{(m)}}, \underline{a}\leq p^m)
P _1 ( x_1, \dots, x _s)
\cdots
Q _r ( \underline{\partial} ^{<\underline{a}> _{(m)}}, \underline{a}\leq p^m)
P _r ( x_1, \dots, x _s),$$
où $P $ est un monôme de $ B _{n}$ et, pour $i=1,\dots, r$,
$P _i $ est un monôme unitaire de $ B _s$ et $Q _i $ est un monôme unitaire $B _{n (p ^m +1)}$, avec toujours 
$\deg M _J \leq c (J+1)$ (en tant que monôme de $B _{N}$).

{\it $3)$ Linéarisation de $P _1 ( x_1, \dots, x _s),\dots, P _r ( x_1, \dots, x _s)$.}
En utilisant \ref{E2} et via \ref{proddagconst}.\ref{proddagconst-2)}, on vérifie que, pour tout $1 \leq u \leq r$,
il existe un vecteur ligne $L ^{(u)} = (l ^{(u)} _1, \dots, l ^{(u)}_s)$ à coefficients dans $A ^\dag$
et, pour tout $1\leq i\leq s$ et tout $j \in \N$,
des polynômes $L ^{(u)} _{ij} \in \pi ^j B _n$ tels que
$\deg (L ^{(u)} _{ij} ) \leq c (j + \deg (P _u))$, $l ^{(u)} _i =\sum _{j\in\N} L ^{(u)} _{ij} (t _1,\dots, t _n)$
et 
$P _u (x _1,\dots, , x _s) = L ^{(u)} X$. Remarquons que pour $\deg (P _u)=0$, on a eu besoin d'avoir $x _{1}=1$ et qu'en fait on peut affiner en remplaçant $\deg (L ^{(u)} _{ij} ) \leq c (j + \deg (P _u))$ par 
$\deg (L ^{(u)} _{ij} ) \leq c (j + \deg (P _u)-1)$.

{\it $4)$ Passage de droite à gauche par rapport à $\underline{\partial} ^{<\underline{a}> _{(m)}}$ des combinaisons linéaires de $x_1,\dots, x _s$ à coefficients dans $A ^\dag$.}
Pour tout $\underline{a} \leq p^m$, avec \cite[2.2.4.(iv)]{Be1}
puis via la formule de Leibnitz de \cite[2.3.4.1]{Be1}, on obtient :
\begin{gather}
  \notag
  \underline{\partial} ^{<\underline{a}> _{(m)}} L ^{(u)} X
=
\sum _{\underline{h} \leq \underline{a}}
\left \{
\begin{smallmatrix}
  \underline{a} \\
  \underline{h}
\end{smallmatrix}
\right \}
\underline{\partial} ^{<\underline{a}-\underline{h}> _{(m)}}
(L ^{(u)} X)
\underline{\partial} ^{<\underline{h}> _{(m)}}
=
\sum _{\underline{h} \leq \underline{a}}
\left \{
\begin{smallmatrix}
  \underline{a} \\
  \underline{h}
\end{smallmatrix}
\right \}
\sum _{\underline{h}' \leq \underline{a}-\underline{h}}
\left \{
\begin{smallmatrix}
  \underline{a}-\underline{h} \\
  \underline{h}'
\end{smallmatrix}
\right \}
\underline{\partial} ^{<\underline{a}-\underline{h}-\underline{h}'> _{(m)}} (L^{(u)})
A ^{(\underline{h}')} X
\underline{\partial} ^{<\underline{h}> _{(m)}}
\\
=
\sum _{\underline{h} \leq \underline{a}}
L ^{(u,\underline{a})} _{\underline{h}} X\underline{\partial} ^{<\underline{h}> _{(m)}},
\end{gather}
où 
$L ^{(u,\underline{a})} _{\underline{h}}= ( l ^{(u,\underline{a})} _{\underline{h},1}, \dots, l ^{(u,\underline{a})} _{\underline{h},1})
:=\left \{
\begin{smallmatrix}
  \underline{a} \\
  \underline{h}
\end{smallmatrix}
\right \}
\sum _{\underline{h}' \leq \underline{a}-\underline{h}}
\left \{
\begin{smallmatrix}
  \underline{a}-\underline{h} \\
  \underline{h}'
\end{smallmatrix}
\right \}
\underline{\partial} ^{<\underline{a}-\underline{h}-\underline{h}'> _{(m)}} (L^{(u)})
A ^{(\underline{h}')} $
est un vecteur ligne à coefficients dans $A ^\dag$.
Or, d'après \ref{proddagconst},
il existe, pour tout $i=1,\dots, s$ et tout $j\in\N$,
des polynômes $L ^{(u,\underline{a})} _{\underline{h},ij}\in  \pi ^j B _n$
tels que $\deg (L ^{(u,\underline{a})} _{\underline{h},ij}) \leq c( j + \deg (P _u) +1)$
et $l ^{(u,\underline{a})} _{\underline{h},i} = \sum _{j\in \N} L ^{(u,\underline{a})} _{\underline{h},ij}(t _1,\dots, t_n)$.

En résumé : le passage de droite à gauche par rapport à $\underline{\partial} ^{<\underline{a}> _{(m)}}$ 
des combinaisons linéaires de $x_1,\dots, x _s$ 
coûte l'ajout de {\og $1$\fg} dans l'inégalité de la forme $\deg (L ^{(u,\underline{a})} _{\underline{h},ij}) \leq c( j + \deg (P _u) +1)$.
Ce nombre $1$ correspond aussi au degré des monômes $\underline{\partial} ^{<\underline{h}> _{(m)}}$ (car $\underline{h} \leq p^m$).

{\it $5)$}
En réitérant le procédé de l'étape $4)$,
on vérifie que $M _J$ est égal à une somme finie de termes de la forme
$$R _{J}=\pi ^J L X
 Q  ( \underline{\partial} ^{<\underline{a}> _{(m)}}, \underline{a}\leq p^m)  $$
 où 
 $Q \in B _{n (p ^m +1)}$ avec $\deg (Q) \leq \deg (Q _1) +\dots + \deg (Q _r)\leq \deg (M _J)$,
  $L=(l _1 ,\dots, l_s)$ est un vecteur ligne à coefficients dans $A ^\dag$ tel que,
pour tout $i =1,\dots, s$ et tout $j\in\N$,
 il existe des polynômes $L _{ij}\in B _{n}$
 tels que
 $\deg (L _{ij}) \leq c (j+ \deg (M _J))$ et
$l _i =\sum _{j\in\N} \pi ^j  L _{ij} (t _1,\dots, t _n)$. 

{\it $6)$ Conclusion.}
Posons $R _{J,j}:= L _{ij} (t _1,\dots, t _n) 
 Q  ( \underline{\partial} ^{<\underline{a}> _{(m)}}, \underline{a}\leq p^m) \in B _{n + n (p ^m +1)} $.
Ainsi, 
$$R _{J}= \sum _{i=1} ^s x _{i} \sum _{j\in\N} \pi ^{J+j} R _{J,j}.$$
Comme $\deg (M _J) \leq \deg (P _J) \leq c (J+1)$, alors
$\deg (Q) \leq c (J+1)$ et $\deg (L _{ij}) \leq c (j+ c (J+1))$.
 D'où :
 $\deg (R _{J,j})
 \leq
 c (j+ c (J+1)) + c (J+1) 
 \leq
 c (1+ c)  (J+j+1)$.

Lorsque $j$ et $J$ sont fixés, l'ensemble des polynômes de la forme $R _{J,j}$ définis comme ci-dessus est de cardinal fini.
Notons $\widetilde{R} _{J,j}$ la somme finie des éléments de cet ensemble. 
On obtient alors la somme
$$z =\sum _{i=1} ^s x _{i} \sum _{J,j\in\N} \pi ^{J+j} \widetilde{R} _{J,j}.$$
Comme $\deg (R _{J,j}) \leq c (1+ c)  (J+j+1)$, il en résulte que 
$ \sum _{J,j\in\N} \pi ^{J+j} \widetilde{R} _{J,j} \in D ^{(m)\dag}$. 
\end{proof}

\begin{prop}
\label{oub-com-otimes}
On garde les notations et hypothèses de \ref{AotimeDdag=DdagU}.
Les morphismes canoniques
$$D ^\dag _{\widehat{\A} ^n _\V}  \otimes _{D ^\dag }
D ^\dag _{U^\dag}
\rightarrow
D ^\dag _{\U },
\
D ^\dag _{U^\dag} \otimes _{D ^\dag }
D ^\dag _{\widehat{\A} ^n _\V}
\rightarrow
D ^\dag _{\U }$$
sont des isomorphismes.
\end{prop}

\begin{proof}
  D'après \ref{AotimeDdag=DdagU},
  $D ^\dag  \otimes _{\V [ \underline{t}] ^\dag }A ^\dag \riso D ^\dag _{U^\dag }$.
D'où:
$$D ^\dag _{\widehat{\A} ^n _\V}  \otimes _{D ^\dag } D ^\dag _{U^\dag}
\liso
D ^\dag _{\widehat{\A} ^n _\V}  \otimes _{D ^\dag }
D ^\dag  \otimes _{\V [ \underline{t}] ^\dag }A ^\dag
\liso
D ^\dag _{\widehat{\A} ^n _\V}
\otimes _{\V [ \underline{t}] ^\dag }
A ^\dag.$$
Or, comme $A ^\dag$ est une $\V [ \underline{t}] ^\dag $-algèbre finie,
$\V \{ \underline{t}\} \otimes _{\V [ \underline{t}] ^\dag } A ^\dag \riso \widehat{A}$.
Donc,
$D ^\dag _{\widehat{\A} ^n _\V}
\otimes _{\V \{ \underline{t}\}  }
\widehat{A}
\riso
D ^\dag _{\widehat{\A} ^n _\V}
\otimes _{\V [ \underline{t}] ^\dag }
A ^\dag$.
L'isomorphisme
$D ^\dag _{\widehat{\A} ^n _\V}
\otimes _{\V \{ \underline{t}\}  }
\widehat{A}
\riso
D ^\dag _{\U}$ de \ref{BerAotimeDdag=DdagU} nous permet de conclure.
\end{proof}

\subsection{Caractérisation des isocristaux surconvergents via les opérateurs différentiels de Mebkhout-Narvaez}

Nous donnons une description des isocristaux surconvergents sur l'espace affine (voir \ref{Eqii)etiii)} et \ref{lemm-Eqii)etiii)}).
Nous en déduisons ensuite, grâce à la section précédente, 
une description des isocristaux surconvergents sur les schémas finis et étales sur l'espace affine (voir \ref{Eq-isoc-Ddag}).

Dans cette section, nous garderons les notations suivantes : 
soient $\PP:= \widehat{\P} ^n _\V$ l'espace projectif formel sur $\V$ de dimension $n$,
  $u _0,\dots ,u _n$ les coordonnées projectives de $\PP$,
  $H _0$ l'hyperplan défini par $u _0=0$, i.e., $H _0:=\P _k ^n \setminus \A ^n _k$.
On désigne par $\O _{\PP} (\hdag H_0) _\Q$ (resp. $\D ^\dag _{\PP} (\hdag H_0) _\Q$) le faisceau des fonctions (resp. opérateurs différentiels de niveau fini) sur $\PP$ à singularités surconvergentes le long de $H _{0}$ (voir \cite[4.2]{Be1}).
On pose de plus $\D  _{\PP} (\hdag H_0) _\Q: = \O _{\PP} (\hdag H_0) _\Q \otimes _{\O _{\PP,\Q}}\D  _{\PP,\Q} $, où 
$\D  _{\PP}$ est le faisceau usuel des opérateurs différentiels sur $\PP$.

D'après le théorème de comparaison de Noot-Huyghe
(voir \cite{Noot-Huyghe-affinite-proj} ou \cite{Noot-Huyghe-affinite-genproj}),
on dispose dans cette situation géométrique de l'isomorphisme :
$D ^{\dag} _{K}=D ^{\dag} _{\A ^{n\dag} _\V, K}
\riso
\Gamma (\PP, \D ^\dag _{\PP} (\hdag H_0) _\Q)$.
Elle établit de plus la formule :
\begin{equation}
  \notag
  \Gamma (\PP, \D ^\dag _{\PP} (\hdag H_0) _\Q)
  =
\left\{ \sum _{\underline{k},\underline{l} \in \N ^n}
a _{\underline{k}, \underline{l}} \underline{t} ^{\underline{k}} \underline{\partial} ^{[\underline{l}]}\,|\,
\exists \eta < 1, \exists c \geq 1 \text{ tels que } |a _{\underline{k}, \underline{l}} |<
c \eta ^{|\underline{k}|+ |\underline{l}|}
\right\},
\end{equation}
où $t _{1}=\frac{u _{1}}{u _{0}},\dots, t _{n}=\frac{u _{n}}{u _{0}}$ désignent les coordonnées canoniques sur l'espace affine.

\begin{vide}
[Théorèmes de type $A$ sur l'espace affine ou un de ses ouverts affines]
\label{TheA-PH0}
Soit $\E$ un $\D ^\dag _{\PP} (\hdag H _0) _\Q$-module cohérent.

$\bullet$ D'après le théorème de type $A$ pour les
$\D ^\dag _{\PP} (\hdag H _0) _\Q$-modules cohérents (voir \cite{Noot-Huyghe-affinite-proj}),
$E:= \Gamma (\PP, \E)$ est un $D ^\dag _K$-module cohérent et
le morphisme canonique
\begin{equation}
\label{TheA-Adag}
\D ^\dag _{\PP} (\hdag H _0) _\Q \otimes _{D ^\dag _K} E \rightarrow \E
\end{equation}
est un isomorphisme.
Ainsi, les foncteurs $\Gamma (\PP, -)$ et $\D ^\dag _{\PP} (\hdag H _0) _\Q \otimes _{D ^\dag _K} -$ induisent des équivalences quasi-inverses entre la catégorie des $\D ^\dag _{\PP} (\hdag H _0) _\Q$-modules cohérents
et celle des $D ^\dag _K$-modules cohérents. 

$\bullet$ De même, 
le foncteur $\Gamma (\PP, -)$ induit une équivalence entre la catégorie des $\D _{\PP} (\hdag H _0) _\Q$-modules cohérents 
(resp. $\O _{\PP} (\hdag H _0) _\Q$-modules cohérents)
et celle des $D _K$-modules cohérents (resp. $\V [\underline{t}] ^\dag _{K}$-modules cohérents).

$\bullet $ Pour tout ouvert affine $\U ' \subset \widehat{\A} ^n _\V$,
d'après le théorème de type $A$ pour les $\D ^\dag  _{\U',\Q}$-modules cohérents
(voir \cite[3.6.5]{Be1}),
le morphisme canonique
\begin{equation}
\label{TheA-Adagbis}
\D ^\dag _{\U',\Q} \otimes _{D ^\dag _{\U',K}} \Gamma (\U', \E)
\rightarrow
\E |\U'
\end{equation}
est un isomorphisme.
En combinant \ref{TheA-Adagbis} et \ref{TheA-Adag}, il en résulte 
que
le morphisme canonique
\begin{equation}
\label{dagdag-dag}
  D ^\dag _{\U',K} \otimes _{D ^\dag _K} E \rightarrow
\Gamma (\U', \E)
\end{equation}
 est un isomorphisme.
\end{vide}

\begin{vide}
\label{Eqii)etiii)}
Grâce à Berthelot (voir \cite[2.2.12]{caro_courbe-nouveau} pour une version écrite d'ailleurs plus forte),
la catégorie $\mathrm{Isoc} ^{\dag} (\A ^{n} _{k}/K)$ est équivalence à celle des $\D ^\dag _{\PP} (\hdag H _0) _\Q$-modules cohérents, 
$\O _{\PP} (\hdag H _0) _\Q$-modules cohérents.
D'après les théorèmes de type $A$ (voir \ref{TheA-PH0}), 
la catégorie $\mathrm{Isoc} ^{\dag} (\A ^{n} _{k}/K)$ est donc équivalence à celle 
des $D ^\dag _K$-modules cohérents, $\V [\underline{t}] ^\dag _{K}$-cohérents.
\end{vide}

\begin{lemm}
\label{lemm-Eqii)etiii)}
Soit $E$ un $D _{K}$-module, cohérent pour sa structure induite de $\V [\underline{t}] ^\dag _{K}$-module.
Les assertions suivantes sont équivalentes :
\begin{enumerate}
\item \label{lemm(i)-Eqii)etiii)} $E$ est un isocristal surconvergent sur $\A ^{n} _{k}$ ;
\item \label{lemm(ii)-Eqii)etiii)} $E$ est muni d'une structure de $D ^\dag _K$-module cohérent prolongeant sa structure de $D _K$-module ;
\item \label{lemm(iii)-Eqii)etiii)} Le morphisme canonique 
$E \to  D ^\dag  _{K} \otimes _{D  _{K}} E$ est un isomorphisme.
\end{enumerate}
\end{lemm}

\begin{proof}
D'après \ref{Eqii)etiii)}, $2\Leftrightarrow 1$.
Par noethéranité de $D  _{K}$, $E$ est $D  _{K}$-cohérent et donc $3\Rightarrow 2$.
Supposons que $E $ soit un isocristal surconvergent sur $\A ^{n} _{k}$.
Notons $\E$ le $\D ^\dag _{\PP} (\hdag H _0) _\Q$-module cohérent, 
$\O _{\PP} (\hdag H _0) _\Q$-module cohérent associé. D'après \cite[2.2.8]{caro_comparaison} (voir les notations de \cite[2.2.2]{caro_comparaison}) le morphisme canonique
$\E \to \D ^\dag _{\PP} (\hdag H _0) _\Q \otimes _{\D_{\PP} (\hdag H _0) _\Q} \E$ est un isomorphisme. 
Il suffit alors d'utiliser les théorèmes de type $A$ pour respectivement les 
$\D ^\dag _{\PP} (\hdag H _0) _\Q$-modules cohérents et les $\D _{\PP} (\hdag H _0) _\Q$-modules cohérents.
\end{proof}

\begin{prop}
\label{pre-Eq-isoc-Ddag}
Soient $g\,:\,U^\dag \rightarrow \A ^{n\dag} _\V$ un morphisme fini étale de
$\V$-algèbres formels faibles lisses et
$A ^\dag := \Gamma (U ^\dag , \O _{U ^\dag})$.
Soit 
$E$ un $D _{U^\dag ,K}$-module, cohérent pour sa structure induite de $A ^\dag _{K}$-module.
Notons $g _{*} (E)$ le $D _{K}$-module $\V [\underline{t}] ^\dag _{K}$-cohérent induit par $E$ (ainsi, $g _{*}$ est le foncteur oubli).
Les assertions suivantes sont équivalentes : 
\begin{enumerate}
\item \label{pre-Eq-isoc-Ddag-i} La connexion de $E$ est surconvergente, i.e., $E$ est un isocristal surconvergent sur $U _{0}$ ;
\item \label{pre-Eq-isoc-Ddag-ii} La connexion de $g _{*} (E)$ est surconvergente, i.e., $g _{*}(E)$ est un isocristal surconvergent sur $\A ^{n} _{k}$ ;
\item \label{pre-Eq-isoc-Ddag-iii} Le morphisme canonique $D _{K}$-linéaire
$g _{*}(E) \to D ^\dag  _{K} \otimes _{D  _{K}}g _{*}(E) $
est un isomorphisme ;
\item \label{pre-Eq-isoc-Ddag-iv} 
Le morphisme canonique $D _{U^\dag ,K}$-linéaire
$E \to D ^\dag  _{U^\dag ,K} \otimes _{D _{U^\dag ,K}}(E )$
est un isomorphisme. 
\end{enumerate}
\end{prop}

\begin{proof}
Par \cite[7.2.15]{LeStum-livreRigCoh}, on vérifie l'équivalence entre les deux premières assertions. 
Il résulte de l'isomorphisme de \ref{lemm-AotimeDdag=DdagU} (avec aussi un passage aux sections globales) que le morphisme canonique $D _{K} \otimes _{\V [ \underline{t}] ^\dag }A ^\dag  \riso D _{U^\dag ,K} $ est un isomorphisme.
Grâce à \ref{AotimeDdag=DdagU2}, il en résulte qu'il en est de même du morphisme canonique : 
$D ^\dag _{K} \otimes _{D _{K}}D _{U^\dag ,K} \rightarrow D ^\dag _{U^\dag ,K}$.
On en déduit l'équivalence entre \ref{pre-Eq-isoc-Ddag-iii} et \ref{pre-Eq-isoc-Ddag-iv}.
Enfin, l'équivalence entre \ref{pre-Eq-isoc-Ddag-ii} et \ref{pre-Eq-isoc-Ddag-iii} 
découle de \ref{lemm-Eqii)etiii)}.
\end{proof}

\begin{coro}
\label{Eq-isoc-Ddag}
Avec les notations de \ref{pre-Eq-isoc-Ddag},
la catégorie $\mathrm{Isoc} ^{\dag} (U _{0}/K)$ des isocristaux surconvergents sur $ U _{0}$
est équivalente à celles des $D ^\dag _{U^\dag ,K}$-modules cohérents, cohérents 
pour leur structure induite de $A ^\dag _{K}$-module.
\end{coro}

\begin{proof}
Cela résulte aussitôt des équivalences \ref{pre-Eq-isoc-Ddag}.\ref{pre-Eq-isoc-Ddag-i} $\Leftrightarrow$ \ref{pre-Eq-isoc-Ddag}.\ref{pre-Eq-isoc-Ddag-iii}$\Leftrightarrow$ \ref{pre-Eq-isoc-Ddag}.\ref{pre-Eq-isoc-Ddag-iv} et \ref{lemm-Eqii)etiii)}.\ref{lemm(ii)-Eqii)etiii)} $\leftrightarrow$ \ref{lemm-Eqii)etiii)}.\ref{lemm(iii)-Eqii)etiii)}.
\end{proof}

\subsection{Caractérisation des isocristaux surcohérents sur certains sous-schémas de l'espace affine}

Dans la suite de cette section, nous conserverons les notations suivantes :
  soient $\PP:= \widehat{\P} ^n _\V$ l'espace projectif formel sur $\V$ de dimension $n$,
  $u _0,\dots ,u _n$ les coordonnées projectives de $\PP$ (ou, par abus de notations, de $\P _\V ^n$ ou $\P _k ^n$),
  $H _0$ l'hyperplan défini par $u _0=0$, i.e., $H _0:=\P _k ^n \setminus \A ^n _k$.
  On note $t _{1}=\frac{u _{1}}{u _{0}},\dots, t _{n}=\frac{u _{n}}{u _{0}}$ les coordonnées canoniques de l'espace affine.
  Soient  $U ^\dag$ un ouvert affine de $ \A ^{n\dag} _\V$,
  $\U$ son complété $p$-adique,
$T _0:=\P _k ^n \setminus U _0$ le diviseur réduit de $\P _k ^n$
  dont le support est le complémentaire de $U _0$.
Soit $v\,:\, Y ^\dag \hookrightarrow U ^\dag$ une immersion fermée
  de $\V$-schémas formels faibles affines et lisses, avec $Y _0$ intègre et $\dim Y _0 = n-r$
  pour un certain entier $r$.
  On suppose de plus qu'il existe un morphisme fini et étale $g_0\,:\, U _0 \rightarrow \A ^n _k $ tel que
  $g _0 (Y _0) \subset \A ^{n-r} _k$( grâce aux travaux de Kedlaya dans \cite{Kedlaya-coveraffine} ou \cite{Kedlaya-coveraffinebis}, cette hypothèse est en fait génériquement valable). On note $g\,:\,U^\dag \rightarrow \A ^{n\dag} _\V$ le relèvement de $g_0$.
Les complétés $p$-adiques de $v$ ou $g$ seront encore notés respectivement $v$ ou $g$.

Le résultat principal de cette section est la caractérisation de \ref{pass-cvasurcv} des isocristaux surcohérents sur $Y _{0}$.

\begin{vide}
Comme $U ^\dag $ est un ouvert de $\A ^{n \dag} _\V$, on obtient le morphisme
canonique de restriction (pour tout niveau $m$) $\mathrm{restr}\,:\, D ^{(m)} _{\A ^{n\dag} _\V} \rightarrow
D ^{(m)} _{U ^\dag}$. Par fonctorialité de la complétion $p$-adique faible puis par passage
à la limite sur le niveau,
il en résulte le morphisme canonique:
\begin{equation}
  \label{dag->Bert}
\mathrm{restr}\,:\, \Gamma (\PP, \D ^\dag _{\PP} (\hdag H_0) _\Q)
\riso
D ^{\dag} _{\A ^{n\dag} _\V}
\rightarrow
D ^{\dag} _{U ^\dag }.
\end{equation}

\end{vide}

\begin{rema}
  \label{conn-finiétal}
  Soit
$h\,:\,X ^{\prime \dag}=\Spf B ^\dag \rightarrow \Spf A ^\dag = X ^{\dag} $
un morphisme fini et étale de $\V$-schémas formels faibles affines et lisses.
Soit $M$ un $B ^\dag _K$-module cohérent muni d'une connexion intégrable
$M \rightarrow M \otimes _{B ^\dag} \Omega ^1 _{B^\dag}$, i.e., un $D _{X ^{\dag}, K}$-module cohérent, $B ^\dag _K$-cohérent. 
Soit
$\mathcal{M}$ un $\widehat{B}  _K$-module cohérent muni d'une connexion intégrable
$\mathcal{M} \rightarrow \mathcal{M} \otimes _{\widehat{B} } \Omega ^1 _{\widehat{B} }$,
i.e., un $D _{\X ', K}$-module cohérent, $\widehat{B}  _K$-cohérent. 
On désigne par $h _{*} (M)$ (resp. $h _{*} (\mathcal{M})$) le $D _{X ^{\dag}, K}$-module cohérent induit via le morphisme 
$D _{X ^{\dag}, K} \to D _{X ^{\prime \dag}, K}$ (resp. $D _{\X , K}\to D _{\X ', K}$)
induit par $h$ (voir \ref{lemm-AotimeDdag=DdagU}). 

On suppose qu'il existe un morphisme $B ^\dag _K$-linéaire $M\to \mathcal{M}$ commutant aux connexions et tel que 
le morphisme canonique induit $\widehat{A}  _K \otimes _{A ^\dag _K} h _{*} (M )\to h _{*} ( \mathcal{M})$ soit un isomorphisme commutant aux connexions.
Dans ces conditions, le morphisme canonique induit $\widehat{B}  _K \otimes _{B ^\dag _K} M \to \mathcal{M}$ est alors un isomorphisme commutant aux connexions, i.e., est $D _{\X ^{\prime}, K}$-linéaire.

\end{rema}

\begin{lemm}
\label{gamma-u!proj}
  Soit $\alpha\,:\,\widehat{\P} ^{n-r} _\V
\hookrightarrow
\widehat{\P} ^{n} _\V$
l'immersion fermée définie par
$u_1=0,\dots, u _r=0$.
Soient $\E$ un $\D ^\dag _{\PP} (\hdag H _0) _\Q$-module cohérent
à support dans $ \P ^{n-r} _k$ et $E := \Gamma (\widehat{\P} ^{n} _\V, \E)$.
Alors,
$\Gamma (\widehat{\P} ^{n-r} _\V, \alpha^! (\E))\riso
\cap _{i =1} ^r \ker (t _{i} \,:\, E\rightarrow E)$.
\end{lemm}

\begin{proof}
  Notons $\PP': =\widehat{\P} ^{n-r} _\V$, $\I$ l'idéal définissant l'immersion fermée $\alpha$.
  On vérifie par complétion $p$-adique
$\widehat{\D} ^{(m)} _{\PP'\rightarrow \PP}\riso
\widehat{\D} ^{(m)} _{\PP}/\I \widehat{\D} ^{(m)} _{\PP}$.
Avec les notations de \cite[1.1.6]{caro_courbe-nouveau},
en ajoutant les singularités surconvergentes le long de $H _{0}$, 
par passage à la limite sur le niveau et tensorisation par $\Q$,
on obtient alors l'isomorphisme
$\D ^{\dag} _{\PP'\rightarrow \PP} (\hdag H _{0}) _{\Q}\riso
\D ^{\dag} _{\PP} (\hdag H _{0}) _{\Q} /\I \D ^{\dag} _{\PP}  (\hdag H _{0}) _{\Q}$.
Cela implique
$$\alpha^! (\E)= \D ^{\dag} _{\PP'\rightarrow \PP} (\hdag H _{0}) _{\Q}
\otimes ^\L _{\alpha^{-1} \D ^{\dag} _{\PP} (\hdag H _{0}) _{\Q}}
\alpha^{-1} \E [-r]
\riso
\O _{\PP} (\hdag H _{0}) _{\Q} /\I \O _{\PP}  (\hdag H _{0}) _{\Q}
\otimes ^\L _{\alpha^{-1} \O _{\PP}  (\hdag H _{0}) _{\Q}}
\alpha^{-1} \E [-r].$$
On remarque que $\I$ est engendré par les sections globales 
  $u _1,\dots, u _r$ et que $\I \O _{\PP}  (\hdag H _{0}) _{\Q}$ est engendré par les sections globales $t _{1}, \dots, t _{r}$ (en effet, $u _{0}$ est inversible dans $\I \O _{\PP}  (\hdag H _{0}) _{\Q}$).
Via la résolution de Koszul induite par la suite régulière des éléments $t _{1}, \dots, t _{r}$ qui engendrent l'idéal $\I \O _{\PP}  (\hdag H _{0}) _{\Q}$ 
de $\O _{\PP}  (\hdag H _{0}) _{\Q}$, 
on calcule
$\mathcal{H} ^0 (\alpha^! (\E))
\riso
\cap _{i =1} ^r \ker (t _{i} \,:\, \E\rightarrow \E)$.
Or, comme
$\E$ est à support dans $\P ^{n-r} _k$,
le théorème de Berthelot-Kashiwara implique
$\mathcal{H} ^0 (\alpha^! (\E)) \riso \alpha^! (\E)$.
D'où
$\alpha^! (\E )\riso
\cap _{i =1} ^r \ker (t_i \,:\, \E\rightarrow \E)$.
On conclut en lui appliquant le foncteur section globale.
\end{proof}

\begin{lemm}
  \label{Gamme!affi}
  Soient $\beta \,:\,\ZZ' \rightarrow \ZZ$ une immersion fermée de $\V$-schémas formels affines et lisses,
  $x _1,\dots, x_r$ des générateurs de l'idéal définissant $\beta$,
  $\E$ un $(F\text{-})\D^\dag _{\ZZ,\Q}$-module cohérent à support dans $Z'$.
  On dispose alors de l'isomorphisme canonique :
$\Gamma (\ZZ', \beta^! (\E)) \riso
\cap _{i=1} ^r \ker (x_i\,:\, \Gamma (\ZZ, \E)  \rightarrow \Gamma (\ZZ, \E)  )$.
\end{lemm}
\begin{proof}
  On procède de manière analogue à la preuve de \ref{gamma-u!proj}.
\end{proof}

\begin{lemm} 
  Soit $\alpha\,:\,\widehat{\P} ^{n-r} _\V
\hookrightarrow
\widehat{\P} ^{n} _\V$
l'immersion fermée définie par
$u_1=0,\dots, u _r=0$. Soit $\beta\,:\,\widehat{\A} ^{n-r} _\V
\hookrightarrow
\widehat{\A} ^{n} _\V$ le morphisme induit par $\alpha$.
Soit $\E$ un $\D ^\dag _{\PP} (\hdag H _0) _\Q$-module cohérent
à support dans $ \P ^{n-r} _k$.
Le diagramme canonique 
\begin{equation}
  \label{Gamme!affi-u!proj}
  \xymatrix {
  {\Gamma (\P ^{n-r} _{k}, \alpha ^{!} (\E))}
  \ar[rr] ^-{\sim} _-{\ref{gamma-u!proj}}
  \ar@{^{(}->}[d]
  &&
  {\cap _{i=1} ^r \ker (t_i\,:\, \Gamma ( \P ^{n} _{k}, \E) \rightarrow \Gamma ( \P ^{n} _{k}, \E))}
  \ar@{^{(}->}[d]
\\
  {\Gamma (\A ^{n-r} _{k}, \alpha ^{!} (\E))}
  \ar@{=}[r] 
    &
  {\Gamma (\A ^{n-r} _{k}, \beta ^{!} (\E |\A ^{n} _{k}))}
    \ar[r] ^-{\sim} _-{\ref{Gamme!affi}}
  & 
  {\cap _{i=1} ^r \ker (t_i\,:\, \Gamma ( \A ^{n} _{k}, \E) \rightarrow
  \Gamma ( \A ^{n} _{k}, \E),}
  }
\end{equation}
où les isomorphismes horizontaux proviennent de \ref{gamma-u!proj} et \ref{Gamme!affi},
est commutatif.
\end{lemm}
\begin{proof}
Cela découle de la construction des isomorphismes horizontaux.
\end{proof}

\begin{vide}
[Quelques équivalences de catégories]
D'après \cite{caro-construction}, on dispose du foncteur pleinement fidèle $\sp _{Y _{0} \hookrightarrow \U,+}$ 
de la catégorie $\mathrm{Isoc} (Y _{0}/K)$ des isocristaux convergents sur $Y _{0}$ dans celle 
des $\D ^\dag _{\U ,\Q} $-modules cohérents à support dans $Y _{0}$. 
Son image essentielle est constituée par les $\D ^\dag _{\U,\Q} $-modules cohérents à support dans $Y _{0}$ tels que 
$v ^! (\G)$ soit $\O _{\Y,\Q}$-cohérent. 

D'après \cite{caro-pleine-fidelite}, on dispose de l'équivalence 
$\sp _{Y ^\dag \hookrightarrow U ^\dag, T _0,+}\,:\,\mathrm{Isoc} ^{\dag} (Y _{0}/K) \cong \mathrm{Isoc} ^{\dag\dag} (Y _{0}/K)$ 
entre la catégorie des isocristaux surconvergents sur $Y _{0}$ et celle des isocristaux surcohérents sur $Y _{0}$.
Cela correspond à une extension sans Frobenius du théorème analogue de \cite{caro-2006-surcoh-surcv}, ce qui nous permet 
d'obtenir le théorème \ref{pass-cvasurcv} qui suit sans structure de Frobenius (et par la même occasion, il est inutile de se préoccuper de la compatibilité à Frobenius des constructions, ce qui simplifie la preuve). 
\end{vide}

\begin{theo}
\label{pass-cvasurcv}
Soit $\E$ un $\D ^\dag _{\PP} (\hdag H _0) _\Q$-module cohérent
  tel que $\E |\U$ soit dans l'image essentielle de $\sp _{Y _{0} \hookrightarrow \U,+}$.
  Il existe alors un isocristal surconvergent $G$ sur $Y _0$ et un isomorphisme
 $\D ^\dag _{\PP} (\hdag T _0) _\Q$-linéaire :
  \begin{equation}
    \notag
    \sp _{Y ^\dag \hookrightarrow U ^\dag, T _0,+} (G)
    \riso
    \E (\hdag T_0).
  \end{equation}
Autrement dit, $\E (\hdag T_0)\in \mathrm{Isoc} ^{\dag \dag} ( Y _{0}/K)$.
De plus, si $\E$ est muni d'une structure de Frobenius, alors $\E (\hdag T_0)$ est surholonome (voir \cite{caro-Tsuzuki-2008}).
\end{theo}

\begin{proof}
Notons
$A ^\dag := \Gamma (U ^\dag , \O _{U ^\dag})$,
$E:= \Gamma (\PP, \E)$ et
$E ' := D ^\dag _{U^\dag} \otimes _{D^\dag } E$ où
l'extension $\mathrm{restr}\,:\,D ^\dag \rightarrow D ^\dag _{U^\dag}$ choisie pour calculer le produit tensoriel
est celle induite par l'immersion ouverte
$U^\dag \subset A ^{n\dag } _\V$, i.e., celle de \ref{dag->Bert}.
Notons $\widetilde{Y} ^\dag := g  ^{-1} (\A ^{n-r \dag} _\V)$,
$a \,:\,\widetilde{Y} ^\dag \rightarrow \A ^{n-r \dag} _\V$ le morphisme fini étale induit par $g$.
Soient
$w \,:\, Y ^\dag \hookrightarrow \widetilde{Y} ^\dag $
(resp. $\widetilde{v}\,:\, \widetilde{Y} ^\dag \hookrightarrow U ^\dag$)
un relèvement de l'immersion fermée
$Y _0\hookrightarrow \widetilde{Y} _0$
(resp. $\widetilde{Y} _0\hookrightarrow U _0$).
Notons $\beta \,:\, \A ^{n-r\dag } _\V \hookrightarrow \A ^{n\dag } _\V$ 
et $\alpha \,:\, \P ^{n-r\dag } _\V \hookrightarrow \P ^{n\dag } _\V$ 
les immersions fermées canoniques.
Les complétés $p$-adiques des morphismes de $\V$-schémas formels faibles lisses seront désignés abusivement par la même lettre, e.g., 
$\beta \,:\, \widehat{\A} ^{n-r} _\V \hookrightarrow \widehat{\A} ^{n} _\V$
ou $\alpha\,:\,\widehat{\P} ^{n-r} _\V
\hookrightarrow
\widehat{\P} ^{n} _\V$. 
Notons enfin $x _1, \dots, x_n$ les coordonnées locales de $U ^{\dag}$ correspondant via $g^*$
à $t _1, \dots, t_n$. 
Posons $G= \cap _{i=1} ^r \ker (x_i\,:\, E' \rightarrow E')$.

\medskip

\textit{$I)$ Le module $G$ correspond à un isocristal surconvergent sur $Y _0$.}

Une des principales difficultés est d'établir que
le module $G$ est un $\Gamma (Y ^\dag , \O _{Y ^\dag,\Q})$-module cohérent.
La stratégie est de se ramener via le morphisme $g$ au cas où la compactification de $Y _0$ dans $P _0$
est lisse.
On construit un $\D ^\dag _{\PP} (\hdag H_0) _\Q$-module cohérent en posant $\FF:= \D ^\dag _{\PP} (\hdag H_0) _\Q \otimes _{D ^\dag _K }g _{*}(E')$, 
où $g _{*}(E')$ désigne le  $D ^\dag _K $-module cohérent induit par $E'$ via $g$ (voir \ref{conn-finiétal}).

$1)$ {\it Vérifions l'isomorphisme $ \FF |\widehat{\A} ^n _\V \riso g_+( \E |\U)$.}\\
D'après \ref{dagdag-dag}, $\Gamma (\widehat{\A} ^n _\V,\FF)
\riso
D ^\dag _{\widehat{\A} ^n _\V,K}  \otimes _{D ^\dag _K} g _{*}(E')$.
Via \ref{oub-com-otimes}, par associativité du produit tensoriel,
il en résulte
$\Gamma (\widehat{\A} ^n _\V,\FF)
=
D ^\dag _{\U,K }
\otimes _{ D ^\dag _{U^\dag,K}}
E'$.
D'un autre côté,
d'après \ref{dagdag-dag}, on dispose de l'isomorphisme canonique :
$D ^\dag _{\U,K }
\otimes _{ D ^\dag _{K}}
E
\riso
\Gamma (\U, \E) $.
D'où :
$D ^\dag _{\U,K }
\otimes _{ D ^\dag _{U^\dag,K}}
E'
\riso
\Gamma (\U, \E) $.
Nous avons ainsi établi l'isomorphisme
$\Gamma (\widehat{\A} ^n _\V,\FF |\widehat{\A} ^n _\V)
\riso
\Gamma (\widehat{\A} ^n _\V,g_*( \E |\U))$.
Or, comme $g$ est fini et étale,
$g _* (\E |\U)\riso g _+ (\E |\U)$ est
un $\D ^\dag _{\widehat{\A} ^n _\V,\Q}$-module cohérent (en effet,
l'image directe par un  morphisme propre conserve la $\D ^{\dag}$-cohérence).
Donc, d'après le théorème de type $A$
sur les $\D ^\dag _{\widehat{\A} ^n _\V,\Q}$-modules cohérents,
$\FF |\widehat{\A} ^n _\V \riso g_*( \E |\U)$.

$2)$ 
Posons $\H:= v ^! (\E |\U)$. 
On obtient un $\D ^\dag _{\widetilde{\Y},\Q}$-module cohérent
$\O _{\widetilde{\Y},\Q}$-cohérent en posant $\widetilde{\H}:=w _+ (\H)$.
En effet, $\H\in \mathrm{Isoc} ^{\dag \dag} (Y_0,Y_0 /K)$, i.e., $\H$ est un $\D ^\dag _{\Y,\Q}$-module cohérent
$\O _{\Y,\Q}$-cohérent. De plus, comme $\dim \widetilde{Y}_0 = \dim Y _0$ et comme $\widetilde{Y}_0$ est lisse,
$Y _0$ est alors une composante connexe de $\widetilde{Y}_0$.
Il en résulte que $\widetilde{\H} \in \mathrm{Isoc} ^{\dag \dag} (\widetilde{Y}_0,\widetilde{Y}_0 /K)$.

$3)$ On dispose de l'isomorphisme $\FF |\widehat{\A} ^n _\V \riso \beta _+ a _+ (\widetilde{\H})$.
En effet, d'après le théorème de Berthelot-Kashiwara, on dispose de l'isomorphisme canonique
$\E |\U \riso v _+ (\H)$.
Or, d'après l'étape $1)$,  $\FF |\widehat{\A} ^n _\V \riso g_+( \E |\U)$. 
On en déduit :
$\FF |\widehat{\A} ^n _\V \riso g_+  v _+ (\H) \riso  g_+  \widetilde{v} _+ (\widetilde{\H})\riso   \beta _+ a _+ (\widetilde{\H})$.

$4)$ {\it Le faisceau $\FF$ est à support dans $\P ^{n-r} _k$.}\\
Notons $H _1,\dots, H _r$ les hyperplans de $\P ^{n} _k$ correspondants à $u _1=0,\dots, u_r=0$.
Il résulte de l'isomorphisme de l'étape $3)$ que $ \FF |\widehat{\A} ^n _\V $ est à support dans $\A ^{n-r}  _{k}$.
Ainsi, pour tout $s =1,\dots, r$, $\FF (\hdag H _s)$ est un $\D ^\dag _{\PP} (\hdag H _s \cup H _0) _\Q$-module cohérent
nul en dehors de $H _s\cup H_0$. Par \cite[4.3.12]{Be1}, on obtient $\FF (\hdag H _s)=0$.
En utilisant le triangle de localisation en $H _{s}$, on en tire $\R \underline{\Gamma} ^\dag _{H _s} (\FF) \riso \FF$.
D'où $\R \underline{\Gamma} ^\dag _{\P ^{n-r} _k} (\FF) \riso \FF$ (voir \cite[2.2.8]{caro_surcoherent}).
Ainsi, $\FF$ est à support dans $\P ^{n-r} _k$.

$5)$ {\it Établissons que $\FF ,\alpha^!(\FF)  \in \mathrm{Isoc} ^{\dag \dag} (\A ^{n-r} _k, \P ^{n-r} _k /K)=\mathrm{Isoc} ^{\dag \dag} (\A ^{n-r} _k /K)$.}\\
D'après l'étape $4)$, il suffit d'établir
que $\alpha^!(\FF)  \in \mathrm{Isoc} ^{\dag \dag} (\A ^{n-r} _k /K)$.
D'après le théorème de Berthelot-Kashiwara, il résulte de l'étape $4)$ que 
$\alpha^!(\FF)$ est un
$\D ^\dag _{\widehat{\P} ^{n-r} _\V } (\hdag H _0 \cap \P ^{n-r} _k ) _{\Q}$-module cohérent
vérifiant
$\alpha_+ \circ \alpha^! (\FF) \riso \FF$.
D'après la caractérisation \cite[2.2.12]{caro_courbe-nouveau},
pour vérifier que $\alpha^!(\FF) \in \mathrm{Isoc} ^{\dag \dag} (\A ^{n-r} _k /K)$, 
il suffit alors d'établir que 
$\alpha^!(\FF)| \widehat{\A} ^{n-r} _\V$ est 
$\O _{ \widehat{\A} ^{n-r} _\V,\Q}$-cohérent.
Or, $\alpha^! (\FF) |\widehat{\A} ^{n-r} _\V
\riso
\beta ^! (\FF |\widehat{\A} ^n _\V )$.
Avec l'étape $3)$, on en déduit alors
$\alpha^! (\FF) |\widehat{\A} ^{n-r} _\V
\riso
\beta ^!  \beta _+ a _+ (\widetilde{\H})
\riso
a _+ (\widetilde{\H})$.
  De plus, comme $a$ est fini et étale, $a _+ (\widetilde{\H})\riso a _* (\widetilde{\H})$.
Or, d'après l'étape $2)$, $\widetilde{\H}$ est un $\O _{\widetilde{\Y},\Q}$-module cohérent.
Il en résulte que $a _+ (\widetilde{\H})$ est en outre $\O _{\widehat{\A} ^{n-r} _\V,\Q}$-cohérent.

$6)$ {\it Le module $G$ est un isocristal surconvergent sur $\widetilde{Y} _0$ et $\Gamma (\widehat{\P} ^{n-r} _\V, \alpha^!(\FF))
=a _{*} (G)$.}\\ 
a) Comme $\alpha^!(\FF) \in \mathrm{Isoc} ^{\dag \dag} (\A ^{n-r} _k /K)$, 
d'après \ref{Eqii)etiii)},
$\Gamma (\widehat{\P} ^{n-r} _\V, \alpha^!(\FF))$ est
un $D ^{\dag} _{\A ^{n-r \dag} _{\V} ,K}$-module cohérent,
$\V [t_{r+1}, \dots ,t_n] ^\dag \otimes _\V K$-cohérent.
\\
b) D'après \ref{gamma-u!proj},
comme $E' = \Gamma ( \widehat{\P} ^{n} _\V, \FF)$,
on obtient
$\Gamma (\widehat{\P} ^{n-r} _\V, \alpha^!(\FF))
=
\cap _{i=1} ^r \ker (t_i\,:\, E' \rightarrow E')
=
\cap _{i=1} ^r \ker (x_i\,:\, E' \rightarrow E')
=G$.
Comme $\Gamma (\widetilde{Y} ^\dag , \O _{\widetilde{Y} ^\dag})$
est une
$\V [ t_{r+1}, \dots, t_n] ^\dag$-algèbre finie,
il en résulte que 
$G$ est un
$\Gamma (\widetilde{Y} ^\dag , \O _{\widetilde{Y} ^\dag,\Q}) _{K}$-module cohérent.
De plus,
comme $E'$ est un $D _{U ^\dag, K}$-module,
par un calcul en coordonnées locales, 
il en résulte que $G$ est muni d'une structure canonique
de $D _{\widetilde{Y} ^\dag, K}$-module telle que, en notant $a _{*} (G)$ le $D _{\A ^{n-r \dag} _{\V} ,K}$-module cohérent induit
via le morphisme canonique $D_{\A ^{n-r \dag} _{\V} ,K}\to D _{\widetilde{Y} ^\dag, K}$ induit par $a$,
l'égalité $\Gamma (\widehat{\P} ^{n-r} _\V, \alpha^!(\FF))
=a _{*} (G)$ ci-dessus est $D _{\A ^{n-r \dag} _{\V} ,K}$-linéaire.
D'après \ref{pre-Eq-isoc-Ddag}, il en résulte que 
$G$ est un $D ^{\dag}_{\widetilde{Y} ^\dag, K}$-module cohérent, $\Gamma (\widetilde{Y} ^\dag , \O _{\widetilde{Y} ^\dag})_{K}$-cohérent.

$7)$ {\it L'isocristal convergent sur $\widetilde{Y} _{0}$ induit par $G$ est isomorphe à
$\Gamma (\widetilde{\Y}, \widetilde{\H})$} (rappelons que d'après l'étape $2)$ 
$\Gamma (\widetilde{\Y}, \widetilde{\H})$ est un isocristal convergent sur $\widetilde{Y} _0$).\\ 
a) D'après \ref{oub-com-otimes}, le morphisme canonique de $(D ^\dag _{\widehat{\A} ^n _\V ,K} , D ^\dag _{U^\dag ,K})$-bimodules
$D ^\dag _{\widehat{\A} ^n _\V ,K}  \otimes _{D ^\dag _K}
D ^\dag _{U^\dag ,K}
\rightarrow
D ^\dag _{\U ,K}$
est un isomorphisme.
On en déduit l'isomorphisme $D ^\dag _{\widehat{\A} ^n _\V,K} $-linéaire :
$D ^\dag _{\U,K }
\otimes _{ D ^\dag _{U^\dag,K}}
E' \liso
D ^\dag _{\widehat{\A} ^n _\V,K}  \otimes _{D ^\dag _K} E' $.
Avec \ref{Gamme!affi} et via le théorème $A$ pour les $\D ^{\dag}_{\widehat{\A} ^{n-r}, \Q}$-modules cohérents
(resp. $\D ^{\dag}_{\widehat{\A} ^{n}, \Q}$-modules cohérents),
en lui appliquant le foncteur $? \mapsto \cap _{i=1} ^r \ker (t_i\,:\, ? \to ?)$, on obtient le morphisme 
$D ^{\dag}_{\widehat{\A} ^{n-r}, K}$-linéaire du bas du diagramme suivant
\begin{equation}
  \label{diag-pass-cvasurcv}
  \xymatrix {
  {G=\cap _{i=1} ^r \ker (x_i\,:\, E' \rightarrow E') }
  \ar@{=}[r]
  \ar@{^{(}->}[d]
  &
  {\cap _{i=1} ^r \ker (t_i\,:\, E' \rightarrow E')}
  \ar@{^{(}->}[d]
\\
  {\cap _{i=1} ^r\ker (x_i\,:\, D ^\dag _{\U,K }
\otimes _{ D ^\dag _{U^\dag,K}}
E' \rightarrow D ^\dag _{\U,K }
\otimes _{ D ^\dag _{U^\dag,K}}
E')}
  &
  {\cap _{i=1} ^r \ker (t_i\,:\, D ^\dag _{\widehat{\A} ^n _\V,K}  \otimes _{D ^\dag _K} E' \rightarrow
  D ^\dag _{\widehat{\A} ^n _\V,K}  \otimes _{D ^\dag _K} E').}
  \ar[l] ^-\sim
  }
\end{equation}
Via un calcul imédiat, ce diagramme est commutatif.\\
b) Le terme en bas à gauche de \ref{diag-pass-cvasurcv}
est canoniquement isomorphe à $\Gamma (\widetilde{\Y}, \widetilde{\H})$.
En effet, 
comme $\widetilde{\H}\riso \widetilde{v} ^! (\E |\U)$,
il résulte de \ref{Gamme!affi} l'isomorphisme :
$\Gamma (\widetilde{\Y}, \widetilde{\H})
\riso
\cap _{i=1} ^r \ker (x_i\,:\,
\Gamma (\U, \E) \rightarrow \Gamma (\U, \E))$.
Or, on a vérifié au cours de la preuve de l'étape $1)$ l'isomorphisme $D ^\dag _{\U,K }
\otimes _{ D ^\dag _{U^\dag,K}}
E'
\riso
\Gamma (\U, \E) $. D'où le résultat. \\
c) On calcule de plus que
la flèche de gauche (resp. de droite) de \ref{diag-pass-cvasurcv} est
$D _{\widetilde{Y} ^\dag, K}$-linéaire (resp. $D _{\A ^{n-r \dag}, K}$-linéaire).\\
d) Via le théorème de type $A$ (plus précisément \ref{dagdag-dag}), 
on vérifie que l'injection
$\Gamma (\widehat{\P} ^{n} _\V, \alpha^!(\FF))
\subset
\Gamma (\widehat{\A} ^{n} _\V, \alpha^!(\FF))$ s'identifie au morphisme canonique
$E' \to D ^\dag _{\widehat{\A} ^n _\V,K}  \otimes _{D ^\dag _K} E'$.
D'après \ref{Gamme!affi-u!proj}, il en résulte que la flèche de droite de \ref{diag-pass-cvasurcv} est isomorphe
à l'injection
$\Gamma (\widehat{\P} ^{n-r} _\V, \alpha^!(\FF))
\subset
\Gamma (\widehat{\A} ^{n-r} _\V, \alpha^!(\FF))$
de l'isocristal surconvergent sur $\A ^{n-r} _k$ associé à $\alpha^!(\FF)$
dans l'isocristal convergent sur $\A ^{n-r} _k$ induit.
Comme l'injection canonique $\Gamma (\widehat{\P} ^{n-r} _\V, \alpha^!(\FF))
\subset
\Gamma (\widehat{\A} ^{n-r} _\V, \alpha^!(\FF))$
induit par extension l'isomorphisme 
$\V \{ t_{r+1}, \dots, t_n\}\otimes _{\V [ t_{r+1}, \dots, t_n] ^\dag}\Gamma (\widehat{\P} ^{n-r} _\V, \alpha^!(\FF))
\riso
\Gamma (\widehat{\A} ^{n-r} _\V, \alpha^!(\FF))$, 
on en déduit que l'injection 
$G \hookrightarrow \Gamma (\widetilde{\Y}, \widetilde{\H})$ (flèche de gauche de \ref{diag-pass-cvasurcv} modulo l'isomorphisme de 7.b))
induit par extension l'isomorphisme
$$\V \{ t_{r+1}, \dots, t_n\}\otimes _{\V [ t_{r+1}, \dots, t_n] ^\dag}G
\riso \Gamma (\widetilde{\Y}, \widetilde{\H}).$$
\\
e) Via a), b), c), d), on conclut grâce à la remarque \ref{conn-finiétal}.

$8)$ {\it Le module $G$ est aussi un isocristal surconvergent sur $Y _0$.}\\
Comme $\widetilde{\H}=w _+ (\H)$, 
il résulte de $7)$ que les restrictions de $G$
sur les composantes connexes $\widetilde{Y} _0$ distinctes de $Y _0$ sont nulles.
Ainsi, $G$ est un isocristal surconvergent sur $Y _0$.

\medskip

\textit{$II)$ Construction de l'isomorphisme $\sp _{Y ^\dag \hookrightarrow U ^\dag, T _0,+} (G)
\riso \E (\hdag T_0)$.}

Soient $\G := \D _{Y^\dag,\Q} \otimes _{D _{Y^\dag,K}} G$,
$v _+ (\G):= v _{*} (\D _{U ^\dag \leftarrow Y^\dag,\Q}
\otimes _{\D _{Y^\dag,\Q}} \G) $
et
$v _+ (G):= \Gamma (U^\dag, \D _{U ^\dag \leftarrow Y^\dag,\Q})
\otimes _{D _{Y^\dag,K}} G $.
Alors, par (passage de droite à gauche de) \cite[2.4.1]{caro_devissge_surcoh},
$v _+ (G) \riso \Gamma(U ^\dag, v _+ (\G))$
et $v_+ (\G) \riso \D _{U^\dag} \otimes _{D _{U^\dag }} v_+(G)$.

Soit $j\,:\, U ^\dag \subset \P ^{n\dag} _\V$ l'immersion ouverte.
On dispose, pour tout $D _{U^\dag }$-module $M$, d'un morphisme
$j_* \D _{U^\dag,\Q} \otimes _{D _{U^\dag ,K}} M \rightarrow
j_* (\D _{U^\dag,\Q} \otimes _{D _{U^\dag ,K}} M )$
fonctoriel en $M$.
Lorsque $M= D _{U^\dag , K}$, celui-ci est un isomorphisme.
En appliquant ces deux foncteurs à une présentation finie de $v_+(G)$,
on obtient un morphisme entre deux présentations finies. Par le lemme des cinq, il en résulte
l'isomorphisme
$j_* \D _{U^\dag,\Q} \otimes _{D _{U^\dag ,K}} v_+(G) \riso
j_* (\D _{U^\dag,\Q} \otimes _{D _{U^\dag ,K}}v_+(G) )$.
Donc,
 $j_* v_+ (\G) \liso j_* \D _{U^\dag,\Q} \otimes _{D _{U^\dag ,K}} v_+(G) $.
D'où :
\begin{equation}
\label{pass-cvasurcv-isopart1}
  \sp _{Y ^\dag \hookrightarrow U ^\dag, T _0,+} (G)=
\D ^\dag _{\PP } (\hdag T _0) _\Q \otimes _{j_* \D _{U^\dag,\Q} } j_* v_+ (\G)
\liso
\D ^\dag _{\PP } (\hdag T _0) _\Q
\otimes _{D _{U^\dag ,K}} v_+(G).
\end{equation}

Posons $\widetilde{\G}: = w _+ (\G)$ l'image directe de $\G$ par $w$ (attention, en tant que $\D$-module et non $\D ^{\dag}$-module), 
$\widetilde{v} _{+} (G):=
\Gamma (U^\dag, \D _{U ^\dag \leftarrow \widetilde{Y}^\dag,\Q})
\otimes _{D _{\widetilde{Y}^\dag,K}} G$.
Soient
$\partial _1, \dots, \partial _n$
les dérivations correspondantes aux coordonnées locales $x _1, \dots, x_n$.
Un calcul classique donne $\widetilde{v} _{+} (G)\riso K [ \partial _1 , \dots, \partial _r ]
\otimes _K G $ et,
comme $G =\cap _{i=1} ^r \ker (x_i\,:\, E' \rightarrow E')$,
on dispose alors du morphisme canonique $D _{U^\dag,K}$-linéaire :
$\widetilde{v} _{+} (G) \rightarrow E'$ (défini par
$\partial _i \otimes x \mapsto \partial _i \cdot x$).

Par transitivité de l'image directe, on obtient $v _+ (\G) \riso \widetilde{v} _+ (\widetilde{\G})$.
Comme $\widetilde{\G}$ est un $\D _{\widetilde{Y}^\dag,\Q} $-module cohérent
tel que
$\Gamma (\widetilde{Y}^\dag,\widetilde{\G}) =
\Gamma (Y^\dag,\G)=G $, on en déduit (grâce à nouveau à  \cite[2.4.1]{caro_devissge_surcoh}) :
$v _+ (G) \riso \widetilde{v} _{+} (G)$.
D'où le morphisme $D _{U^\dag,K}$-linéaire :
$v _{+} (G) \rightarrow E'$.

Il en dérive le morphisme $D ^\dag _{U^\dag,K}$-linéaire :
$D ^\dag _{U^\dag,K} \otimes _{D _{U^\dag,K}} v _+ (G) \rightarrow E'=
D ^\dag _{U^\dag,K} \otimes _{D^\dag _K} E$.
Or, on bénéficie d'après \cite[2.7.3.(ii)]{huyghe-comparaison}
de l'injection canonique:
$j_* \D ^\dag _{U^\dag,\Q} \rightarrow \D ^\dag _{\PP } (\hdag T _0) _\Q$.
D'où:
$\Gamma (U ^\dag, \D ^\dag _{U^\dag,\Q} ) \rightarrow
\Gamma (\PP, \D ^\dag _{\PP } (\hdag T _0) _\Q)$.
Comme on dispose des morphismes canoniques
$D ^\dag _{U^\dag,\Q}  \rightarrow
\Gamma (U ^\dag, \D ^\dag _{U^\dag,\Q} ) $
et
$\Gamma (\PP, \D ^\dag _{\PP } (\hdag T _0) _\Q)
\rightarrow
\D ^\dag _{\PP } (\hdag T _0) _\Q$, il en dérive par composition:
$D ^\dag _{U^\dag,\Q}  \rightarrow
\D ^\dag _{\PP } (\hdag T _0) _\Q$.
En appliquant le foncteur $\D ^\dag _{\PP } (\hdag T _0) _\Q \otimes _{D ^\dag _{U^\dag,K} }-$
à
$D ^\dag _{U^\dag,K} \otimes _{D _{U^\dag,K}} v _+ (G) \rightarrow 
D ^\dag _{U^\dag,K} \otimes _{D^\dag _K} E$,
on obtient:
$\D ^\dag _{\PP } (\hdag T _0) _\Q \otimes _{D _{U^\dag,K}} v _+ (G)
\rightarrow
\D ^\dag _{\PP } (\hdag T _0) _\Q
\otimes _{D^\dag _K} E$.
Or, 
$
\D ^\dag _{\PP } (\hdag T _0) _\Q \otimes _{D^\dag _K} E \riso
\D ^\dag _{\PP } (\hdag T _0) _\Q
\otimes _{\D ^\dag _{\PP } (\hdag H _0) _\Q} \D ^\dag _{\PP } (\hdag H _0) _\Q
\otimes _{D^\dag _K} E
\riso
\D ^\dag _{\PP } (\hdag T _0) _\Q
\otimes _{\D ^\dag _{\PP } (\hdag H _0) _\Q} \E
=
\E (\hdag T_0)$, le dernier isomorphisme résultant de 
\ref{TheA-Adag}.
D'où :
\begin{equation}
\label{pass-cvasurcv-isopart2}
\D ^\dag _{\PP } (\hdag T _0) _\Q \otimes _{D _{U^\dag,K}} v _+ (G)
\rightarrow
\E (\hdag T_0) .
\end{equation}

En composant \ref{pass-cvasurcv-isopart1} et \ref{pass-cvasurcv-isopart2}, on obtient  
le morphisme canonique 
$\phi\,:\,\sp _{Y ^\dag \hookrightarrow U ^\dag, T _0,+} (G)
\rightarrow \E (\hdag T_0)$.
Via \cite[5.2.4]{caro_devissge_surcoh}, on vérifie que
$\phi$ est un isomorphisme en dehors de $T _0$.
D'après \cite[4.3.12]{Be1}, comme
$\phi$
est un morphisme de
$\D ^\dag _{\PP} (\hdag T_0) _\Q$-modules cohérents,
il en résulte que
$\phi$ est un isomorphisme.

\end{proof}

\begin{rema}
D'après \cite[2.7.3.(ii)]{huyghe-comparaison}
on dispose de l'injection canonique:
$j_* \D ^\dag _{U^\dag,\Q} \rightarrow \D ^\dag _{\PP } (\hdag T _0) _\Q$. Comme $T _{0}$ n'est pas forcément ample, on ne peut directement
conclure de \cite{huyghe-comparaison} que cette injection induit par passage aux sections globales un isomorphisme.
De plus, il n'est pas sûr que l'on puisse dans ces conditions disposer de théorème $A$ pour les $\D ^\dag _{\PP } (\hdag T _0) _\Q$-modules cohérents.
Comme la preuve de \ref{pass-cvasurcv} est de passer aux sections globales pour travailler avec les faisceaux des opérateurs différentiels à la Mebkhout-Narvaez, à moins de disposer d'une idée nouvelle, l'hypothèse que le $\D ^\dag _{\PP } (\hdag T _0) _\Q$-module cohérent $\E (\hdag T _{0})$ provienne par extension d'un $\D ^\dag _{\PP } (\hdag H _0) _\Q$-module cohérent est indispensable (ce qui est important en fait est que $H _{0}$ soit ample).
\end{rema}

\section{Stabilité de l'holonomie}

\subsection{Résolution des conjectures de Berthelot sur la stabilité de l'holonomie pour les $\V$-schémas formels projectifs et lisses}
Le théorème \ref{conjD-proj} ci-après signifie que la conjecture
\cite[5.3.6.D)]{Beintro2} de Berthelot est vérifiée lorsque le diviseur est ample.
\begin{theo}
  \label{conjD-proj}
  Soient $\PP$ un $\V$-schéma formel propre et lisse,
  $H _0$ un diviseur ample de $P _0$, $\mathfrak{A}$ l'ouvert de $\PP$ complémentaire de $H _0$.
  Soit $ \E \in F \text{-}D ^\mathrm{b} _\mathrm{coh} (\D ^\dag _{\PP} (\hdag H _0) _\Q)$ tel que 
  $\E |\mathfrak{A} \in F \text{-}D ^\mathrm{b} _\mathrm{hol} (\D ^\dag _{\mathfrak{A},\Q})$.
  Alors $\E \in F \text{-}D ^\mathrm{b} _\mathrm{surhol} (\D ^\dag _{\PP,\Q})$.
\end{theo}

\begin{proof}

Il existe une immersion fermée $\alpha _{0}\,:\,P _0 \hookrightarrow \P ^n _k$ telle que
$(\P ^n _k \setminus \A _k ^n) \cap P _0 = H _0$.
D'après le théorème de Berthelot-Kashiwara 
$\alpha _0 ^! \circ \alpha _{0+} (\E) \riso \E$. De plus $\E  |\mathfrak{A}$ 
est holonome si et seulement si $\alpha _{0+} (\E) $ (pour la version holonome du théorème de Berthelot-Kashiwara, voir par exemple \cite[1.14]{caro_surholonome}).
Comme $\alpha _{0+} (\E)$ est un $F \text{-}\D ^\dag _{\widehat{\P} ^n _\V} (\hdag \P ^n _k \setminus \A _k ^n) _\Q$-module cohérent tel que 
$\E |\widehat{\A} ^n _\V$ est un $F \text{-}\D ^\dag _{\widehat{\A} ^n _\V,\Q}$-module holonome,
comme la surholonomie est préservée par image inverse extraordinaire,
on se ramène ainsi au cas où $\PP= \widehat{\P} ^n _\V$ et $H _0=\P ^n _k \setminus \A _k ^n $.

Nous procédons à présent par récurrence sur l'ordre lexicographique
  $(\dim \mathrm{Supp} (\E), N _{\mathrm{cmax}})$,
  où $\dim \mathrm{Supp} (\E)$ désigne la dimension du support de $\E$ et
  $N _{\mathrm{cmax}}$ signifie le nombre de composantes irréductibles du support de $\E$
  dont la dimension vaut $\dim \mathrm{Supp} (\E)$, i.e., de dimension maximale.
  Le cas où $\dim \mathrm{Supp} (\E)\leq 1$ est déjà connu (voir \cite[2.3.15]{caro-Tsuzuki-2008}).
  Supposons donc $\dim \mathrm{Supp} (\E)\geq 1$.

Pour tout entier $j$, $\mathcal{H} ^{j} (\E)$ est un $F \text{-}\D ^\dag _{\PP} (\hdag H _0) _\Q$-module cohérent tel que 
$\E |\mathfrak{A}$ est un $F \text{-}\D ^\dag _{\mathfrak{A},\Q}$-module holonome. 
Or, pour établir que $\E \in F \text{-}D ^\mathrm{b} _\mathrm{surhol} (\D ^\dag _{\PP,\Q})$, il suffit 
de vérifier que, pour tout entier $j$, $\mathcal{H} ^{j} (\E)$ est un $\D ^\dag _{\PP,\Q}$-module surholonome. 
  On se ramène ainsi à supposer que le complexe $\E$ est réduit à un terme.

Notons $X _0 $ le support de $\E$.
Soit $\U$ un ouvert affine de $\PP$ inclus dans $\mathfrak{A}$
tel que $Y_0:= X_0 \cap U_0$
soit intègre, lisse et dense dans une composante irréductible de $X$ de dimension $\dim X$.
Grâce à Elkik (voir \cite{elkik}, il existe un relèvement $v \,:\, \Y \hookrightarrow \U$
de l'immersion fermée $Y _0 \hookrightarrow U _0$.
D'après le théorème de Berthelot-Kashiwara,
comme $\E |\U$ est holonome et à support dans $Y _0$,
$v ^! (\E |\U)$ est un $F$-$\D ^\dag _{\Y,\Q}$-module holonome.
Par \cite[5.3.5.(i)]{Beintro2},
quitte à rétrécir $\U$ et $\Y$,
on peut supposer que
$v ^! (\E |\U)$ est $\O _{\Y,\Q}$-cohérent.
Notons $T _0 $ le diviseur réduit de $P _0 $ complémentaire de $U _0$ (voir \cite[1.3.1]{caro-2006-surcoh-surcv}).
On peut supposer $\U$ muni de coordonnées locales $x _1,\dots, x _n$ telles que
$\Y$ soit défini par l'idéal engendré par $x_1,\dots, x_r$.
En outre, via \cite[Theorem 2]{Kedlaya-coveraffinebis}
(appliqué au point $0$ et avec les diviseurs irréductibles définis par $x_1=0,\dots, x _r=0$),
quitte à nouveau à rétrécir $U _0$,
on peut supposer qu'il existe un morphisme fini et étale $g_0\,:\, U _0 \rightarrow \A ^n _k $ tel que
  $g _0 (Y _0) \subset \A ^{n-r} _k$.
Grâce au théorème \ref{pass-cvasurcv},
il en résulte que
$\E (\hdag T _0)$ est surholonome.
On conclut la récurrence en utilisant le triangle de localisation
de $\E$ en $T _0$.
\end{proof}

\begin{rema}
\label{ss-Frob-rema}
Le théorème \ref{conjD-proj} est faux si le complexe $\E$ n'est pas muni d'une structure de Frobenius (pour une version sans structure de Frobenius 
de la notion d'holonomie, on pourra consulter \cite{caro-pleine-fidelite}), même en remplaçant l'hypothèse 
{\og $\E |\mathfrak{A} \in F \text{-}D ^\mathrm{b} _\mathrm{hol} (\D ^\dag _{\mathfrak{A},\Q})$ \fg}
par 
{\og $\E |\mathfrak{A} \in F \text{-}D ^\mathrm{b} _\mathrm{surhol} (\D ^\dag _{\mathfrak{A},\Q})$ \fg}. 
En effet, pour ce genre de contre-exemple (i.e., que se passe sans structure de Frobenius ?), il s'agit de revenir
à l'exemple donné tout à la fin de \cite{Be1} par Berthelot ; situation où 
$\PP = \widehat{\P} ^{1} _{V}$ et $A= \mathbb{G} _{m,k}$.
D'après cet exemple, il existe un $\D ^\dag _{\PP} (\hdag H _0) _\Q$-module cohérent $\E$
tel que $\E |\mathfrak{A}$ soit $\O _{\mathfrak{A},\Q}$-cohérent et donc $\D ^\dag _{\mathfrak{A},\Q}$-holonome
et $\D ^\dag _{\mathfrak{A},\Q}$-surholonome. Par contre, ce module $\E$ n'est même pas $\D ^\dag _{\PP,\Q}$-cohérent.

Avec ce même contre-exemple, le théorème \ref{conjD-proj-extr} est faux sans structure de Frobenius.
\end{rema}

\begin{coro}
  \label{coro-conjD-proj}
  Soient $\PP$ un $\V$-schéma formel propre et lisse,
  $H _0$ un diviseur ample de $P _0$, $\mathfrak{A}$ l'ouvert de $\PP$ complémentaire de $H _0$,
  $Y _0$ un sous-schéma fermé lisse de $A _0$.

   Il existe une équivalence entre
  la catégorie $F\text{-}\mathrm{Isoc} ^\dag (Y _0/K)$
  des $F$-isocristaux surconvergents sur $Y _0$
  et la catégorie des
  $ F \text{-}\D ^\dag _{\PP} (\hdag H _0) _\Q$-modules cohérents $\E$ tels
   que $\E |\mathfrak{A}$ soit dans l'image essentielle de $\sp _{Y_0 \hookrightarrow \mathfrak{A},+}$.
\end{coro}

\begin{proof}
D'après \cite[2.3.1]{caro-2006-surcoh-surcv}, on dispose d'une équivalence entre la catégorie
  des $F$-isocristaux surconvergents sur $Y _0$ et
  celle des $F$-isocristaux surcohérents sur $Y _0$.
  Or, d'après \cite{caro_surholonome}, les modules dans l'image de $\sp _{Y_0 \hookrightarrow \mathfrak{A},+}$ sont holonomes. 
Par \ref{conjD-proj}, cela implique qu'un $ F \text{-}\D ^\dag _{\PP} (\hdag H _0) _\Q$-module cohérent $\E$ tel
   que $\E |\mathfrak{A}$ soit dans l'image essentielle de $\sp _{Y_0 \hookrightarrow \mathfrak{A},+}$
   est un $F$-isocristal surcohérent sur $Y _{0}$. La réciproque est immédiate.

\end{proof}

Le théorème \ref{conjD-proj} reste valable en remplaçant {\og holonome \fg} par {\og à fibres extraordinaires finies\fg} (voir la définition \cite[2.1]{caro_surholonome} ou \cite[1.3.1]{caro_caract-surcoh}) :
\begin{theo}
  \label{conjD-proj-extr}
  Soient $\PP$ un $\V$-schéma formel propre et lisse,
  $H _0$ un diviseur ample de $P _0$, $\mathfrak{A}$ l'ouvert de $\PP$ complémentaire de $H _0$ et
  $ \E \in F \text{-}D ^\mathrm{b} _\mathrm{coh} (\D ^\dag _{\PP} (\hdag H _0) _\Q)$.
  Si $\E |\mathfrak{A}$ est à fibres extraordinaires finies
  alors $\E \in F \text{-}D ^\mathrm{b} _\mathrm{surhol} (\D ^\dag _{\PP,\Q})$.
\end{theo}

\begin{proof}
  La preuve est analogue à celle de \ref{conjD-proj} mais avec
  quelques modifications : on se ramène de même
  au cas où $P _0 = \P ^n _k$ et $H _0=\P ^n _k \setminus \A _k ^n $.
  On procède toujours par récurrence sur l'ordre lexicographique
  $(\dim \mathrm{Supp} (\E), N _{\mathrm{cmax}})$. Le cas $\dim \mathrm{Supp} (\E)\leq 1$ résulte toujours
\cite[2.3.15]{caro-Tsuzuki-2008}. La différence ici est que l'on ne peut pas directement se ramener au cas où $\E$ est un module
  car la propriété d'être à fibres extraordinaires finies n'est pas {\og a priori\fg} vérifiée pour les
  espaces de cohomologie $\H ^l (\E)$, $l\in \Z$.
Notons $X _0 $ le support de $\E$.
En remplaçant \cite[5.3.5.(i)]{Beintro2} par \cite[1.3.4]{caro_caract-surcoh} (avec l'égalité \cite[1.3.2]{caro_caract-surcoh}),
on vérifie de manière analogue (à la preuve de \ref{conjD-proj})
qu'il existe un ouvert affine $\U$ de $\PP$ inclus dans $\mathfrak{A}$
tel que $Y_0:= X_0 \cap U_0$
soit intègre, lisse et dense dans une composante irréductible de $X$ de dimension $\dim X$
et tel que les espaces de cohomologie de
$v ^! (\E |\U)$ sont $\O _{\Y,\Q}$-cohérents.
Or, d'après le théorème de Berthelot-Kashiwara, pour tout entier $r$,
$v ^!(\mathcal{H} ^r (\E) |\U) \riso \mathcal{H} ^r  (v ^!(\E |\U) )$, qui
est donc $\O _{\Y,\Q}$-cohérent.
Notons $T _0 $ le diviseur réduit de $P _0 $ complémentaire de $U _0$.
En utilisant \cite[Theorem 2]{Kedlaya-coveraffinebis},
quitte à nouveau à rétrécir $U _0$,
on peut supposer qu'il existe un morphisme fini et étale $g_0\,:\, U _0 \rightarrow \A ^n _k $ tel que
  $g _0 (Y _0) \subset \A ^{n-r} _k$.
Via le théorème \ref{pass-cvasurcv},
il en résulte que
$(\mathcal{H} ^r \E )(\hdag T _0)$ est surholonome.
Comme $(\mathcal{H} ^r \E )(\hdag T _0) \riso
\mathcal{H} ^r  (\E (\hdag T _0))$,
le complexe $\E (\hdag T _0)$ est donc surholonome.
On conclut la récurrence en utilisant le triangle de localisation
de $\E$ en $T _0$.
\end{proof}

\begin{theo}
  \label{holo=surholo-proj}
  Soient $\PP$ un $\V$-schéma formel projectif et lisse,
  $\E \in F\text{-}D ^\mathrm{b} _\mathrm{coh} (\D ^\dag _{\PP,\Q})$.
  Les assertions suivantes sont équivalentes :
  \begin{enumerate}
    \item \label{1} Le $F\text{-}$complexe $\E$ appartient à $F\text{-}D ^\mathrm{b} _\mathrm{hol} (\D ^\dag _{\PP,\Q})$.
    \item \label{2} Le $F\text{-}$complexe $\E$ est à fibres extraordinaires finies.
    \item \label{3} Le $F\text{-}$complexe $\E$ appartient à $F\text{-}D ^\mathrm{b} _\mathrm{surcoh} (\D ^\dag _{\PP,\Q})$.
    \item \label{4} Le $F\text{-}$complexe $\E$ appartient à $F\text{-}D ^\mathrm{b} _\mathrm{surhol} (\D ^\dag _{\PP,\Q})$.
  \end{enumerate}
\end{theo}

\begin{proof}
D'après \cite{caro-Tsuzuki-2008}, on sait déjà $\ref{4} \Leftrightarrow \ref{3} $.
L'implication $\ref{3}\Rightarrow \ref{2}$ est claire.
Supposons que $\E \in F\text{-}D ^\mathrm{b} _\mathrm{surcoh} (\D ^\dag _{\PP,\Q})$.
Pour tout entier $j$, $\mathcal{H} ^{j} (\E)$ est $ F\text{-}\D ^\dag _{\PP,\Q}$-module surcohérent, 
et donc surholonome avec $\ref{4} \Leftrightarrow \ref{3} $. 
D'après \cite[2.5]{caro_surholonome},
on en déduit que $\mathcal{H} ^{j} (\E)$ est holonome. 
D'où l'implication
  $\ref{3}\Rightarrow \ref{1}$.
  Prouvons à présent $\ref{1} \Rightarrow \ref{4} $.
  Supposons donc $\E \in F\text{-}D ^\mathrm{b} _\mathrm{hol} (\D ^\dag _{\PP,\Q})$.
  Comme $\PP$ est projectif, d'après le théorème de Berthelot-Kashiwara,
  on se ramène au cas où $\PP= \widehat{\P} ^n _\V$.
Soit $\H $ l'hyperplan de $\widehat{\P} ^n _\V$ défini par $u_0=0$, i.e. $H_0 :=\P ^n _k \setminus \A _k ^n$.
Alors, d'après \ref{conjD-proj},
$\E (\hdag H _0)$ est surholonome.
Via le triangle de localisation de $\E$ en $H_0$,
il en résulte que $\R \underline{\Gamma} ^\dag _{H _0} (\E)$ est aussi holonome.
En notant $\alpha\,:\, \H \hookrightarrow \widehat{\P} ^n _\V$ l'immersion fermée canonique,
on obtient $\alpha ^! (\E)\riso \alpha ^!(\R \underline{\Gamma} ^\dag _{H _0} (\E))$.
D'après la version holonome du théorème de Berthelot-Kashiwara (voir \cite[1.14]{caro_surholonome}),
on en déduit que $\alpha ^! (\E)$ est holonome.
En procédant par récurrence sur $n$,
on obtient alors la surholonomie de $\alpha ^! (\E)$.
D'où la surholonomie de $\R \underline{\Gamma} ^\dag _{H _0} (\E)\riso \alpha _+ \alpha ^! (\E)$.
Via le triangle de localisation de $\E$ en $H_0$,
il en résulte que $\E$ est aussi surholonome.
Nous avons donc prouvé $\ref{1} \Rightarrow \ref{4} $.
Enfin, pour établir l'implication $\ref{2} \Rightarrow \ref{4} $
on procède de manière analogue à la preuve de $\ref{1} \Rightarrow \ref{4} $
modulo le remplacement de l'utilisation du théorème \ref{conjD-proj}
par \ref{conjD-proj-extr}.
\end{proof}

D'après la stabilité de la surholonomie (voir \cite{caro_surholonome}), on obtient
le corollaire suivant qui répond positivement dans le cas projectif aux conjectures
\cite[5.3.6.A),B)]{Beintro2} de Berthelot.
\begin{coro}
\label{coro}
Soient $f\,:\, \PP ' \rightarrow \PP$ un morphisme de $\V$-schémas formels projectifs lisses,
$\E \in F\text{-}D ^\mathrm{b} _\mathrm{hol} (\D ^\dag _{\PP,\Q})$,
$\E '\in F\text{-}D ^\mathrm{b} _\mathrm{hol} (\D ^\dag _{\PP',\Q})$.
Alors $f _+ (\E') \in F\text{-}D ^\mathrm{b} _\mathrm{hol} (\D ^\dag _{\PP,\Q})$
et $f ^! (\E) \in F\text{-}D ^\mathrm{b} _\mathrm{hol} (\D ^\dag _{\PP',\Q})$.
\end{coro}

\subsection{Stabilité de l'holonomie par les six opérations de Grothendieck sur les $k$-variétés quasi-projectives}

\begin{nota}
\label{defi-hol-PTX}
Soit $Y$ une variété sur $k$. 
On suppose qu'il existe
 un $\V$-schéma formel $\PP$ projectif et lisse, un diviseur $T$ de $P$, un sous-schéma fermé $X$ de $P$ tels que $Y := X \setminus T$.
On notera $F\text{-}\mathrm{Hol} (\PP, T, X/K)$ (resp. $F\text{-}\mathrm{Surhol} (\PP, T, X/K)$) la catégorie des 
$F\text{-}\D ^\dag _{\PP,\Q}$-modules holonomes (resp. $F\text{-}\D ^\dag _{\PP,\Q}$-modules surholonomes) $\E$ à support dans $X$ tels que le morphisme canonique $\E \to \E (\hdag T)$ soit un isomorphisme. 
Il résulte de \ref{holo=surholo-proj} que 
$F\text{-}\mathrm{Hol} (\PP, T, X/K)=F\text{-}\mathrm{Surhol} (\PP, T, X/K)$.
Or, d'après \cite[4.15]{caro_surholonome}, la catégorie $F\text{-}\mathrm{Surhol} (\PP, T, X/K)$ 
ne dépend à isomorphisme canonique près que de $Y/K$, i.e., ne dépend pas de tels $\PP$, $X$, $T$ tels que $Y= X \setminus T$.
Il en est alors de même de $F\text{-}\mathrm{Hol} (\PP, T, X/K)$. On la notera simplement
$F\text{-}\mathrm{Hol} (Y/K)$ ou $F\text{-}\mathrm{Surhol} (Y/K)$. Notons que l'on définit sans structure de Frobenius
$\mathrm{Surhol} (Y/K)$, alors que $\mathrm{Hol} (Y/K)$ n'a pas de sens a priori (car on ne sait pas si l'holonomie sans structure de Frobenius est stable). 
\end{nota}

\begin{rema}
\label{local-proj-plong}
Soit $Y$ une $k$-variété affine. 
Choisissons une immersion fermée de la forme
$Y \hookrightarrow \A ^n _k$. 
En notant $X$ l'adhérence de $Y$ dans $\P ^{n} _{k}$ et $T$ le diviseur de $\P ^n _k$
complémentaire de $\A ^n _k$, on obtient $Y =X \setminus T$.
Les hypothèses de \ref{defi-hol-PTX} sont donc satisfaites pour les schémas affines.
\end{rema}

\begin{vide}
[$F$-$\D$-modules holonomes sur une $k$-variété]

Soit $Y$ une variété sur $k$.  
Lorsqu'il n'existe pas de $\V$-schéma formel $\PP$ projectif et lisse, de diviseur $T$ de $P$, de sous-schéma fermé $X$ de $P$ tels que $Y =X \setminus T$, 
d'après \cite[5.2]{caro_surholonome}, il est encore possible de construire la catégorie $F\text{-}\mathrm{Surhol} (Y /K)$ en procédant par recollement (la situation est localement vérifiée grâce à \ref{local-proj-plong}).
Avec cette remarque, on définit alors la catégorie $F\text{-}\mathrm{Hol} (Y /K)$ de manière analogue et on obtient 
$F\text{-}\mathrm{Hol} (Y /K)=F\text{-}\mathrm{Surhol} (Y /K)$.

Pour la commodité du lecteur, donnons brièvement sa construction : 
avec \ref{local-proj-plong}, choisissons un recouvrement fini ouvert $(Y _\alpha) _{\alpha \in \Lambda}$ de $Y$ tel que,
pour tout $\alpha\in \Lambda$, il existe
un $\V$-schéma formel projectif et lisse $\PP _\alpha$,
  un sous-schéma fermé $X _\alpha$ de
  $P _\alpha$ et un diviseur $T_\alpha$ de $P_\alpha$ tels que $Y _\alpha = X _\alpha\setminus T_\alpha$.
  Pour tous $\alpha, \beta$ et $\gamma \in \Lambda$,
  on note $p _1 ^{{\alpha \beta}}$ : $\PP _\alpha \times _\S \PP _\beta \rightarrow \PP _\alpha$ et
  $p _2 ^{{\alpha \beta}}$ : $\PP _\alpha \times _\S \PP _\beta \rightarrow \PP _\beta$
  les projections canoniques, $X _{{\alpha \beta}}$
  l'adhérence schématique de
  $Y _\alpha \cap Y _\beta$ dans $P _\alpha \times P _\beta$ (via l'immersion
  $Y _\alpha \cap Y _\beta \hookrightarrow Y _\alpha \times Y _\beta
  \hookrightarrow P _\alpha \times P _\beta$),
  $T _{{\alpha \beta}}=
 ( \smash{p _1 ^{\alpha \beta }}) ^{-1} (T _\alpha) \cup (\smash{p _2 ^{\alpha \beta }}) ^{-1} (T _\beta)$.

  Soient
  $j _1 ^{\alpha \beta} $ : $Y _{\alpha}  \cap Y _{\beta} \subset Y _\alpha$,
  $j _2 ^{\alpha \beta} $ : $Y _{\alpha}  \cap Y _{\beta} \subset Y _\beta$,
  $j _{12} ^{\alpha \beta \gamma} $ : $Y _{\alpha}  \cap Y _{\beta}\cap Y _{\gamma} \subset
  Y _{\alpha}  \cap Y _{\beta}$,
  $j _{23} ^{\alpha \beta \gamma} $ :
  $Y _{\alpha}  \cap Y _{\beta}\cap Y _{\gamma}\subset Y _\beta \cap Y _\gamma$
  et
  $j _{13} ^{\alpha \beta \gamma} $ : $Y _{\alpha } \cap Y _{\beta}  \cap Y _{\gamma}
  \subset
  Y _\alpha \cap Y _\gamma$
  les immersions ouvertes.
Pour $i=1,2$, posons $j _i ^{\alpha \beta !}:=\R \underline{\Gamma} ^\dag _{X _{{\alpha \beta}}}
\circ (\hdag T _{{\alpha \beta}})\circ p _{i} ^{\alpha \beta !} $.
On définit de même les foncteurs $j _{12}^{\alpha \beta \gamma !} $, $j _{13}^{\alpha \beta \gamma !} $, $j _{23}^{\alpha \beta \gamma !} $.

  On définit la catégorie $F\text{-}\mathrm{Hol} (Y,\, (Y _\alpha, \PP _\alpha, T_\alpha,  X_\alpha) _{\alpha \in \Lambda}/K)$
  de la façon suivante :
\begin{itemize}
\item   Un objet est constitué par la donnée, pour tout $\alpha \in \Lambda$,
  d'un objet $\E _\alpha$ de
  $F\text{-}\mathrm{Hol} (\PP _\alpha, T _\alpha, X _\alpha /K)$ et,
  pour tous $\alpha,\beta \in \Lambda$,
  d'un isomorphisme dans
  $F\text{-}\mathrm{Hol} (\PP _\alpha \times \PP _\beta , T _{\alpha \beta}, X _{\alpha \beta}/K)$,
  $\theta _{\alpha \beta}$ :
  $j _2 ^{\alpha \beta !} (\E _\beta)
  \riso j _1 ^{\alpha \beta !} (\E _\alpha)$ ;
ces isomorphismes vérifiant 
la condition de cocycle
$j _{13}
^{\alpha \beta \gamma !} (\theta _{  \alpha \gamma} )=
j _{12} ^{\alpha \beta \gamma !} (\theta _{  \alpha \beta} )
\circ j _{23} ^{\alpha \beta \gamma !} ( \theta _{ \beta \gamma })$.

La famille d'isomorphismes $(\theta _{\alpha \beta}) _{\alpha ,\beta \in \Lambda}$
est appelée {\it donnée de recollement} de $ (\E _\alpha) _{\alpha \in \Lambda} $.

\item Les flèches
$ ((\E _\alpha) _{\alpha \in \Lambda}, (\theta _{  \alpha \gamma}) _{\alpha,\beta  \in \Lambda}) \rightarrow
  ((\E '_\alpha) _{\alpha \in \Lambda}, (\theta '_{  \alpha \gamma}) _{\alpha,\beta  \in \Lambda})$
de la catégorie
$F\text{-}\mathrm{Hol}
  (Y,\, (Y _\alpha, \PP _\alpha, T_\alpha,  X_\alpha) _{\alpha \in \Lambda}/K)$
  sont les familles de morphismes
  $ \E _\alpha \rightarrow   \E '_\alpha$
  commutant aux données de recollement respectives.
\end{itemize}

La catégorie
$F\text{-}\mathrm{Hol}
  (Y,\, (Y _\alpha, \PP _\alpha, T_\alpha,  X_\alpha) _{\alpha \in \Lambda}/K)$
  ne dépend pas du choix de la famille
$(Y _\alpha, \PP _\alpha, T_\alpha,  X_\alpha) _{\alpha \in \Lambda}$
et se notera
$F\text{-}\mathrm{Hol} (Y /K)$.
Ses objets sont les $F$-$\D$-modules arithmétiques holonomes sur $Y$.

\end{vide}

\begin{vide}
[Complexes à cohomologie bornée et $F$-holonome de $\D$-modules arithmétiques sur une $k$-variété quasi-projective]
Soit $Y$ une variété quasi-projective sur $k$. 
Il existe alors un $\V$-schéma formel $\PP$ projectif et lisse, deux sous-schémas fermés $T$, $X$ de $P$ tels que $Y = X \setminus T$.
 On note $F\text{-}D ^\mathrm{b} _\mathrm{hol}  (\PP, T, X/K)$ (resp. $(F\text{-})D ^\mathrm{b} _\mathrm{surhol}  (\PP, T, X/K)$) la sous-catégorie pleine 
 de $F\text{-}D ^\mathrm{b} _\mathrm{hol} (\D ^\dag _{\PP,\Q})$ 
 (resp. $(F\text{-})D ^\mathrm{b} _\mathrm{surhol} (\D ^\dag _{\PP,\Q})$)
  des complexes $\E$  à support dans $X$ tels que le morphisme $\E \to \E (\hdag T)$ canonique soit un isomorphisme.
Or, d'après \cite[4.18]{caro_surholonome}, la catégorie $(F\text{-})D ^\mathrm{b} _\mathrm{surhol}  (\PP, T, X/K)$ est indépendante, à isomorphisme canonique près, d'un tel choix de $\PP$, $X$, $T$ tels que $Y= X \setminus T$. On la note $(F\text{-})D ^\mathrm{b} _\mathrm{surhol}  (Y/K)$.
D'après \ref{holo=surholo-proj}, on obtient
 $F\text{-}D ^\mathrm{b} _\mathrm{hol}  (\PP, T, X/K)=F\text{-}D ^\mathrm{b} _\mathrm{surhol}  (\PP, T, X/K)$.
 On note donc simplement $F\text{-}D ^\mathrm{b} _\mathrm{hol}  (Y/K)$ à la place de $F\text{-}D ^\mathrm{b} _\mathrm{hol}  (\PP, T, X/K)$. 
 Ainsi, $F\text{-}D ^\mathrm{b} _\mathrm{hol}  (Y/K)=F\text{-}D ^\mathrm{b} _\mathrm{surhol}  (Y/K)$. On désigne par $D ^\mathrm{b} _\mathrm{F\text{-}hol}  (Y/K)$ l'image essentielle 
 du foncteur oubli $F\text{-}D ^\mathrm{b} _\mathrm{hol}  (Y/K) \to D ^\mathrm{b} _\mathrm{surhol}  (Y/K)$.
\end{vide}

 \begin{vide}
[Opérations cohomologiques sur les $k$-variétés quasi-projectives]
Soit $b\,:\, Y' \to Y$ un morphisme de variétés quasi-projectives sur $k$. 
Par \cite[4.19]{caro_surholonome} (il y a une faute de frappe, il faut d'après la remarque \cite[4.22]{caro_surholonome} enlever a priori la structure de Frobenius), 
on dispose alors des foncteurs 
$b _{+},b _{!}\,:\, D ^\mathrm{b} _\mathrm{surhol}  (Y'/K) \to D ^\mathrm{b} _\mathrm{surhol}  (Y/K)$
appelés respectivement image directe et image directe extraordinaire par $b$.
On en déduit la factorisation: 
$b _{+},b _{!}\,:\, D ^\mathrm{b} _\mathrm{F\text{-}hol}  (Y'/K) \to D ^\mathrm{b} _\mathrm{F\text{-}hol}  (Y/K)$.
De même, avec \cite[4.19]{caro_surholonome}, \cite[4.22]{caro_surholonome}, \cite{caro-Tsuzuki-2008} (ou \cite{caro-stab-prod-tens}),
on dispose des foncteurs 
$b ^{+},b ^{!}\,:\, D ^\mathrm{b} _\mathrm{F\text{-}hol}  (Y/K) \to D ^\mathrm{b} _\mathrm{F\text{-}hol}  (Y'/K)$
appelés respectivement image inverse et image inverse extraordinaire par $b$,
du foncteur dual $\DD _{Y}\,:\, D ^\mathrm{b} _\mathrm{F\text{-}hol}  (Y/K) \to D ^\mathrm{b} _\mathrm{F\text{-}hol}  (Y/K)$
et du produit tensoriel $-\otimes _{\O _{Y}}-\,:\, D ^\mathrm{b} _\mathrm{F\text{-}hol}  (Y/K)\times D ^\mathrm{b} _\mathrm{F\text{-}hol}  (Y/K) \to D ^\mathrm{b} _\mathrm{F\text{-}hol}  (Y/K)$.
\end{vide}

\bibliographystyle{smfalpha}

\providecommand{\bysame}{\leavevmode ---\ }
\providecommand{\og}{``}
\providecommand{\fg}{''}
\providecommand{\smfandname}{et}
\providecommand{\smfedsname}{\'eds.}
\providecommand{\smfedname}{\'ed.}
\providecommand{\smfmastersthesisname}{M\'emoire}
\providecommand{\smfphdthesisname}{Th\`ese}

\bigskip
\noindent Daniel Caro\\
Laboratoire de Mathématiques Nicolas Oresme\\
Université de Caen
Campus 2\\
14032 Caen Cedex\\
France.\\
email: daniel.caro@math.unicaen.fr

\end{document}